\documentclass[preprint]{imsart}

\RequirePackage{amsthm,amsmath,amssymb}
\RequirePackage{graphicx,appendix}
\RequirePackage[numbers]{natbib}
\RequirePackage[colorlinks,citecolor=blue,urlcolor=blue]{hyperref}

\numberwithin{equation}{section}
\theoremstyle{plain}
\newtheorem{theorem}{Theorem}[section]
\newtheorem{lemma}[theorem]{Lemma}

\theoremstyle{remark}
\newtheorem{remark}[theorem]{Remark}

\newcommand{\Tr}{{\bf Tr}}

\begin{document}

\begin{frontmatter}
\title{The Spectrum of Random Inner-product Kernel Matrices}
\runtitle{Inner-product Random Kernel Matrices}

\begin{aug}
\author{\fnms{Xiuyuan} \snm{Cheng}
\ead[label=e1]{xiuyuanc@math.princeton.edu}}
\and
\author{\fnms{Amit} \snm{Singer}
\ead[label=e2]{amits@math.princeton.edu}}
\runauthor{X. Cheng and A. Singer}
\affiliation{Princeton University}
\address{Princeton University, \\
Fine Hall, Washington Rd.,\\
Princeton, NJ, 08544\\
}
\end{aug}

\begin{abstract}
We consider $n$-by-$n$ matrices whose $(i,j)$-th entry is $f(X_{i}^{T}X_{j})$, where $X_1,\ldots,X_n$ are i.i.d. standard Gaussian random vectors in $\mathbb{R}^{p}$, and $f$ is a real-valued function. The eigenvalue distribution of these random kernel matrices is studied at the ``large $p$, large $n$" regime. It is shown that, when $p,n \to \infty$ and $p/n = \gamma $ which is a constant, and $f$ is properly scaled so that $Var(f(X_{i}^{T}X_{j}))$ is $\mathcal{O}(p^{-1})$, the spectral density converges weakly to a limiting density on $\mathbb{R}$. The limiting density is dictated by a cubic equation involving its Stieltjes transform. While for smooth kernel functions the limiting spectral density has been previously shown to be the Marcenko-Pastur distribution, our analysis is applicable to non-smooth kernel functions, resulting in a new family of limiting densities.

\end{abstract}



\end{frontmatter}

\section{Introduction}

In recent years there has been significant progress in the development and application of kernel methods in machine learning and statistical analysis of high-dimensional data \cite{scholkopf2002learning}. These methods include kernel PCA (Principal Component Analysis), the ``kernel trick" in SVM (Support Vector Machine), and non-linear dimensionality reduction \cite{BelkinNiyogi, lafon2006}, to name a few. In such kernel methods, the input is a set of $n$ high-dimensional data points $X_1,\ldots, X_n$ from which an $n$-by-$n$ matrix is constructed, where its $(i,j)$-th entry is a symmetric function of $X_i$ and $X_j$. Whenever the function depends merely on the inner-product $X_i^TX_j$, it is called an inner-product kernel matrix.

In this paper we study the spectral properties of an $n \times n$ symmetric random kernel matrix $A$ whose construction is as follows. Let $X_1,\ldots, X_n$ be $n$ i.i.d Gaussian random vectors in $\mathbb{R}^p$, where $X_i \sim \mathcal{N}(0,p^{-1}I_p)$ and $I_p$ is the $p\times p$ identity matrix. That is, the $np$-many coordinates $\{(X_{i})_{j},1\le i\le n,1\le j\le p\}$ are i.i.d Gaussian random variables with mean 0 and variance $p^{-1}$. The entries of $A$ are defined as
\begin{equation}
A_{ij}=\begin{cases}
f(X_{i}^{T}X_{j};p), & i\neq j,\\
0, & i=j,
\end{cases}
\label{eq:kernalmatrixA}
\end{equation}
where $f(\xi;p)$ is a real-valued function possibly depending on $p$.
We will later consider another model where $X_i$ are drawn from the uniform distribution over the unit sphere $S^{p-1}$ in $\mathbb{R}^p$.

The study of the spectrum of large random matrices, since Wigner's semi-circle law, has been an active research area motivated by applications such as quantum physics, signal processing, numerical linear algebra, statistical inference, among others. An important result is the Marcenko-Pastur (M.P.) law \cite{marchenko1967distribution} for the spectrum of random matrices of the form $S = XX^T$ (also known as Wishart matrices), where $X$ is a $p$-by-$n$ (complex or real) matrix with i.i.d Gaussian entries. In the  ``large $p$, large $n$" limit, i.e. $p,n\rightarrow\infty$ and $p/n=\gamma$ ($0 < \gamma < \infty$), the spectral density of $S$ converges to a deterministic limit, known as the Marcenko-Pastur distribution, which has $\gamma$ as its only parameter. We refer the reader to \cite{Anderson2010book}, \cite{Tao2011book} and \cite[Chapters 1-3]{bai} for an introduction of these topics. Notice that Wishart matrices share the non-zero eigenvalues with their corresponding Gram matrices $G=X^TX$, the latter of which, neglecting the difference at the diagonal entries, can be considered as a
kernel matrix as in Eqn. (\ref{eq:kernalmatrixA})
with the linear kernel function $f(\xi;p)=\xi$.
Thus, the M.P. law and other results involving Wishart matrices
can be translated to the Gram matrix case.

The spectrum of inner-product random kernel matrices with kernel functions that are locally smooth at the origin has been studied in \cite{karoui10}. It was shown that, in the limit $p,n\rightarrow\infty$ and $p/n=\gamma$, 
\begin{itemize}
\item[(1)]
whenever $f$ is locally $C^{3}$, the non-linear kernel matrix converges asymptotically in spectral norm to a linear kernel matrix;
\item[(2)] 
with less regularity of $f$ (locally $C^{2}$), the weak convergence of the spectral density is established.
\end{itemize} We refer to \cite{karoui10}
and references therein for more details, including a complete review of the origins of this problem. The problem we study here is similar to the one considered in ~\cite{karoui10}, except that we allow the kernel function $f$ to belong to a much larger class of functions, in particular, $f$ can be discontinuous at the origin.

Our main result, Thm \ref{thm:main}, establishes the convergence of the spectral density of random kernel matrices under the condition that the kernel function belongs to a weighted $L^2$ space, is properly normalized and satisfies some additional technical conditions. The limiting spectral density is characterized by an algebraic equation, Eqn. (\ref{eq:mzeqn}), of its Stieltjes transform. The equation involves only three parameters, namely $\nu$, $a$ and $\gamma$. The parameter $\nu$ is the limit of $p\cdot Var(f(X_i^TX_j))$ and simply scales the limiting spectral density. The parameter $a$ is the limiting coefficient of the linear term $\xi$ in the expansion of $f(\xi;p)$ into rescaled Hermite polynomials, and has some non-trivial effect on the shape of the limiting spectral density. The result concerning the weak convergence of the spectral density in \cite{karoui10} can be regarded as a special case of our result. Specifically, \cite{karoui10} proves that for a locally smooth kernel function, the limiting spectral density is dictated by its first-order Taylor expansion. The linear term in our rescaled Hermite expansion asymptotically coincides with the first-order term of the Taylor expansion. See also Remark \ref{rk:smoothf} after Thm. \ref{thm:main}.

Notice that the entries of the random kernel matrix are dependent. For example, the triplet of entries $(i,j)$, $(j,k)$ and $(k,i)$ are mutually dependent. 
In the literature of random matrix theory (RMT),
random matrices with dependent entries have received
some attention.
For example, the spectral distribution of random matrices with ``finite-range" dependency among entries is studied in \cite{Anderson2008}.
However, we do not find studies of this sort to be readily applicable to the analysis of the random inner-product kernel matrices considered here. 
We emphasize that our result only addresses the weak limit of the spectral density, while leaving many other questions about random kernel matrices unanswered. These include the analysis of the local statistics of the eigenvalues, the limiting distribution of the largest eigenvalue, and universality type questions with respect to different probability distributions for the data points. 

The rest of the paper is organized as follows:
in Sec. \ref{sec:quickreview} we review the definition and properties of the Stieltjes transform (Sec. \ref{sub:reviewmz}), and revisit the proof of the M.P. law using the Stieltjes transform (Sec. \ref{subsec:A1isMP}). Sec. \ref{sec:L2f} includes the statement of our main theorem, Thm. \ref{thm:main}, and the result of some numerical experiments. The proof of Thm. \ref{thm:main} is established in Sec. \ref{sec:proof}. Finally, the concluding remarks, discussion and open problems are provided in Sec. \ref{sec:discussion}.

Notations: For a vector $X$, we denote by $|X|$ its $l^{2}$ norm, i.e. for $X=(X_{1},\cdots,X_{p})^{T}$ in $\mathbb{R}^{p}$, $|X|=\sqrt{X_{1}^{2}+\cdots+X_{p}^{2}}$. We write $x={\cal O}(1)p^{\alpha}$ to indicate that $|x|\le Cp^{\alpha}$ for some positive constant $C$ and large enough $p$ (which also implies large enough $n$ since  $p/n=\gamma$).  Also, ${\cal O}_{a}(1)$ means that the constant $C$ depends on the quantity $a$, and the latter is often independent of $p$. Throughout the paper, $\zeta$ stands for a random variable observing the standard normal distribution.

\section{Review of the Stieltjes Transform and the M.P. Law}\label{sec:quickreview}

\subsection{The Stieltjes Transform}\label{sub:reviewmz} 

For a probability measure $d\mu$ on $\mathbb{R}$, its Stieltjes
transform (also known as the Cauchy transform) is defined as (see, e.g. Appendix B of \cite{bai})
\[
m(z)=\int_{\mathbb{R}}\frac{1}{t-z}d\mu(t),\quad \Im(z)>0,
\]
 and hence $\Im(m)>0$. The probability density function can be recovered from its Stieltjes transform via the ``inversion formula" 
\begin{equation}\label{eq:inversionformula}
\lim_{b\rightarrow0+}\frac{1}{\pi}\Im(m(t+ib))=\frac{d\mu}{dt}(t),
\end{equation}
 where the convergence is in the weak sense.

Point-wise convergence of the Stieltjes transform implies weak convergence of the probability density (Thm. B.9 in \cite{bai}). This is the fundamental tool that we use to establish the main result in our paper. For the $n$-by-$n$ random kernel matrix $A$, its empirical spectral density (ESD) is defined as 
\begin{equation}\label{eq:ESD}
ESD_{A}=\frac{1}{n}\sum_{i=1}^{n}\delta_{\lambda_{i}(A)}(\lambda)d\lambda,
\end{equation}
where $\{\lambda_{i}(A),i=1,\cdots,n\}$ are the $n$ (real) eigenvalues of $A$.  Considering $ESD_{A}$ as a random probability measure on $\mathbb{R}$, we have its Stieltjes transform as
 \begin{equation}\label{eq:mAdef}
m_{A}(z)=\frac{1}{n}\sum_{i=1}^{n}\frac{1}{\lambda_{i}(A)-z}=\frac{1}{n}\mathbf{Tr}(A-zI)^{-1},\quad\Im(z)>0.
\end{equation}
To show the convergence of $ESD_{A}$, in expectation (or in a.s. sense), it suffices to show that, for every fixed $z$ above the real axis, $m_{A}(z)$ converges to the Stieltjes transform of the limiting density in expectation (or in a.s. sense).

Another convenience brought by fixed $z$ is the uniform boundedness of many quantities. Specifically, for $z=u+iv$, $v>0$, \[
|m_{A}(z)|\le\frac{1}{n}\sum_{i=1}^{n}\frac{1}{|\lambda_{i}(A)-z|}\le\frac{1}{n}\sum_{i=1}^{n}\frac{1}{v}=\frac{1}{v}.
\]
Also,
\begin{equation}
\left|\left((A-zI)^{-1}\right)_{ii}\right|\le\frac{1}{v},\quad1\le i\le n,
\label{eq:boundii}
\end{equation}
 which follows from the spectral decomposition of $A$.

\subsection{Proving the M.P. Law using the Stieltjes Transform}
\label{subsec:A1isMP}

Thm. \ref{thm:A1isMP} is the version of the M.P. law for random kernel matrices with a linear kernel function. 
The version for Wishart matrices is well known and its proof can be found in many places, see e.g. \cite[Chapter 3.3]{bai}. 

\begin{theorem}[the M.P. law for random linear-kernel matrices]\label{thm:A1isMP}
Suppose that $X_{i}\sim \mathcal{N}(0,p^{-1}I_{p})$. Let $A$ be the random kernel matrix as in Eq. (\ref{eq:kernalmatrixA}) with the kernel function $f(\xi)=a\xi$ where $a$ is a constant. Then the limiting spectral density of $A$ is 
\begin{equation}
\rho_{I}(t)=\frac{1}{a}\rho_{M.P.}\left(\frac{t+a}{a};\frac{1}{\gamma}  \right).
\label{eq:rho1}
\end{equation}
The density function $\rho_{M.P.}(t; y)$, with positive constant $y$ as a parameter, is defined as 
\begin{equation}
\rho_{M.P.}(t;y)=\left(1-\frac{1}{y}\right)^{+}\delta_{0}(t)+\frac{\sqrt{(b(y)-t)^{+}(t-a(y))^{+}}}{2\pi y t},
\label{eq:mp0}
\end{equation}
 where $(x)^{+} = \max \{ x, 0 \}$, $b(y)=(1+\sqrt{y})^{2}$
and $a(y)=(1-\sqrt{y})^{2}$. The convergence of $ESD_{A}$ to $\rho_I(t)dt$ 
is in the weak sense, almost surely. 
\end{theorem}
 
\begin{remark}
In Eq. (\ref{eq:rho1}), the rescaling by $a$ is due to the constant $a$ in front of the inner-product, and the shifting by $a$ is due to our setting diagonal entries to be zero. Also, Eq. (\ref{eq:mp0}) is slightly different from the M.P. distribution in literature, since the random kernel matrices that we consider are $n$-by-$n$ and the variance of $X_{i}^{T}X_{j}$ is $p^{-1}$, while Wishart matrices are $p$-by-$p$ and have a different normalization.
\end{remark}

\begin{remark} \label{rk:A1_2}
The distribution of the largest eigenvalue (i.e. the spectral norm, denoted as $s(A)$), a question independent from the limiting spectral density, is well-understood for Wishart matrices, and thus applies to Gram matrices. It has been shown that the largest eigenvalue converges almost surely to its mean value, following a stronger result about the limiting distribution of the largest eigenvalue, namely the Tracy-Widom Law \cite{johnstone2001distribution}. The Tracy-Widom Law of the largest eigenvalue has been shown to be universal for certain sample covariance matrices with non-Gaussian entries, see e.g. \cite{soshnikov2002note, karoui07}. 
As a result (the smallest eigenvalue of a Wishart matrix is always non-negative), as $p,n\to \infty, p/n=\gamma$, almost surely $s(A)<b(\gamma^{-1})+1$, which is an ${\cal O}(1)$ constant depending on $\gamma$ only. 
\end{remark}

Another way to characterize Eq. (\ref{eq:rho1}) is that $m_I(z)$, the Stieltjes transform of $\rho_I(t)$, satisfies the following quadratic equation 
\begin{equation}
-\frac{1}{m(z)}=z+a\left(1-\frac{1}{1+\frac{a}{\gamma}m(z)}\right).
\label{eq:mz1}
\end{equation}
In the literature, Eq. (\ref{eq:mz1}) is sometimes called the M.P. equation. 
It has been shown that  Eq. (\ref{eq:mz1}) has a unique solution with positive imaginary part (Lemma 3.11 of \cite{bai}).

We reproduce the proof of Thm. \ref{thm:A1isMP} here, since some key techniques will be used in proving our main result. 

\begin{proof}[Proof of Thm. \ref{thm:A1isMP}]
In two steps it can be shown that $m_A(z)$, as defined in Eq. (\ref{eq:mAdef}), converges almost surely to the solution of Eq. (\ref{eq:mz1}). Without loss of generality, let $a=1$. 

{\bf Step 1. Reduce a.s. convergence to convergence of $\mathbb{E}m_{A}(z)$.}

\begin{lemma}[concentration of $m_A$ at $\mathbb{E}m_A$]\label{lemma:m_A-Em_A}
 For the $n$-by-$n$ random kernel matrix $A$ as in Eq. (\ref{eq:kernalmatrixA}), 
 where $X_{i}$'s are independent random vectors, 
 and a fixed complex number $z$ with $\Im (z) >0$ , we have that as $n\rightarrow\infty$, 
 \[
m_{A}(z)-\mathbb{E}m_{A}(z)\rightarrow0
\]
 almost surely, and also 
 \begin{equation}
\mathbb{E}|m_{A}-\mathbb{E}m_{A}| \le{\cal O}(1)n^{-1/2}. 
\label{eq:Lemma.|m_A-Em_A|<n^-1/2}
\end{equation}
\end{lemma}

The above lemma relies on that $\Im(z) > 0$ and that the $X_i$'s are independent, while there is no restriction on the specific form of the kernel function, nor on the distribution of $X_i$. The proof  (left to Appendix \ref{sec:A2}) uses a martingale inequality, combined with the observation that among all the entries of $A$ only the $k$-th column/row depend on $X_k$.

{\bf Step 2. Convergence of $\mathbb{E}m_{A}(z)$.} Observe that 
\begin{align*}
\mathbb{E}m_{A}(z) & =\mathbb{E}\frac{1}{n}\mathbf{Tr}(A-zI)^{-1}\\
 & =\mathbb{E}\frac{1}{n}\sum_{i=1}^{n}\left((A-zI)^{-1}\right)_{ii}\\
 & =\mathbb{E}\left((A-zI)^{-1}\right)_{nn},
\end{align*}
 where the last equality follows from that  the rows/columns of $A$ are exchangeable and so are those of $(A-zI)^{-1}$. We then need the following formula
\begin{equation}
((A-zI)^{-1})_{nn}=\frac{1}{(A_{nn}-z)-A_{\cdot,n}^{T}(A^{(n)}-zI_{n-1})^{-1}A_{\cdot,n}},
\label{eq:inv_nn}
\end{equation}
 where $A^{(n)}$ is the top left $(n-1)\times(n-1)$ minor of $A$, i.e.
the matrix $A$ is written in blocks as 
\[
A=\left[\begin{array}{cc}
A^{(n)} & A_{\cdot,n}\\
A_{\cdot,n}^{T} & A_{nn}
\end{array}\right],
\]
 and $I_{n-1}$ is the $(n-1)\times(n-1)$ identity matrix. 
 Notice that since $\Im(z) > 0$, both $A-zI$ and $A^{(n)}-zI_{n-1}$
are invertible. Formula (\ref{eq:inv_nn}) can be verified by elementary 
linear algebra manipulation. 

By Eq. (\ref{eq:inv_nn}) (recall that $A_{nn} = 0$ from Eq. (\ref{eq:kernalmatrixA})),
 \begin{equation}
\mathbb{E}m_{A}(z)=\mathbb{E}\left((A-zI)^{-1}\right)_{nn}=\mathbb{E}\frac{1}{-z-A_{\cdot,n}^{T}(A^{(n)}-zI_{n-1})^{-1}A_{\cdot,n}}.
\label{eq:EmA1}
\end{equation}
To proceed, we condition on the choice of $X_{n}$, and write
\begin{equation}
X_{i}=\eta_{i}(X_{n})_{0}+\tilde{X}_{i},\quad1\le i\le n-1,\label{eq:Xn_1}
\end{equation}
 where $(X_{n})_{0}=\frac{X_{n}}{|X_{n}|}$ is the unit vector in
the same direction of $X_{n}$, and $\tilde{X}_{i}$ lie in the $(p-1)$-dimensional subspace orthogonal to $X_{n}$. Due to the orthogonal invariance of the standard multivariate Gaussian distribution,
we know that $\eta_{i} \sim \mathcal{N}(0,p^{-1})$, $\tilde{X}_i \sim \mathcal{N}(0, p^{-1}I_{p-1})$, and they are independent.  Now we have 
\begin{equation}
X_{i}^{T}X_{n}=\eta_{i}|X_{n}|,\quad1\le i\le n-1,
\label{eq:Xn_2}
\end{equation}
 and 
\begin{equation}
X_{i}^{T}X_{j}=\eta_{i}\eta_{j}+\tilde{X}_{i}^{T}\tilde{X}_{j},\quad1\le i,j\le n-1,i\neq j.
\label{eq:Xn_3}
\end{equation}
Define $\eta=(\eta_{1},\cdots,\eta_{n-1})^{T}$, $D_{\eta}=diag\{\eta_{1}^{2},\cdots,\eta_{n-1}^{2}\}$ which is a diagonal matrix. Also, define 
\begin{equation}
\tilde{A}_{ij}^{(n)}=\begin{cases}
\tilde{X}_{i}^{T}\tilde{X}_{j}, & i\ne j,\\
0, & i=j,
\end{cases}\quad 1\le i,j\le n-1.
\label{eq:tildeA1}
\end{equation}
 Then 
\begin{align}
 A_{\cdot,n}^{T}(A^{(n)}-zI_{n-1})^{-1}A_{\cdot,n} 
 & =|X_{n}|^{2}\eta^{T}(\eta\eta^{T}-D_{\eta}+\tilde{A}^{(n)}-zI_{n-1})^{-1}\eta\nonumber \\
 & =|X_{n}|^{2}\cdot\frac{\eta^{T}(\tilde{A}^{(n)}-D_{\eta}-zI_{n-1})^{-1}\eta}{1+\eta^{T}(\tilde{A}^{(n)}-D_{\eta}-zI_{n-1})^{-1}\eta} \nonumber \\
 & =|X_{n}|^{2}\left(1-\frac{1}{1+\eta^{T}(\tilde{A}^{(n)}-D_{\eta}-zI_{n-1})^{-1}\eta}\right),
 \label{eq: A1denominator} 
\end{align}
 where to get the 2nd line we use the \textit{Sherman-Morrison} formula
\[
q^{T}(pq^{T}+M-zI)^{-1}=\frac{q^T(M-zI)^{-1}}{1+q^{T}(M-zI)^{-1}p}, \quad \forall p,q.
\]

By showing that the denominator in Eq. (\ref{eq: A1denominator}) is asymptotically concentrating at the value of $\mathbb{E}\tilde{m}(z)$, where  $\tilde{m}(z): =\frac{1}{n}\Tr(\tilde{A}^{(n)}-zI_{n-1})^{-1}$, we end up with
\[
\mathbb{E}\left|m_{A}(z)-\left(-z-\left(1-\left(1+\frac{1}{\gamma}\mathbb{E}\tilde{m}(z)\right)^{-1}\right)\right)^{-1}\right|  \to  0.
\]
The detailed derivation is left to Appendix (Lemma \ref{lemma:E|m_A-RHS(tildem)|}). Notice that the probability law of $\eta_{i}$ and $\tilde{X}_{i}^T\tilde{X}_{j}$ do not depend on the position of $X_{n}$, so we omit the conditioning on $X_{n}$ when computing the probabilities and expectations.  Furthermore, by Lemma \ref{bound:E|m_A-m_tildeA|}, 
\begin{equation}
\mathbb{E}|m_{A}(z)-\tilde{m}(z)| \to 0,
\label{eq:mtildem}
\end{equation}
thus 
\begin{equation}\label{eq:Etildemto0}
\mathbb{E}\tilde{m}(z)-\left(-z-\left(1-\left(1+\frac{1}{\gamma}\mathbb{E}\tilde{m}(z)\right)^{-1}\right)\right)^{-1} \to 0.
\end{equation}
Since the quadratic Eq. (\ref{eq:mz1}) has a unique solution $m_{I}(z)$ with positive imaginary part, Eq. (\ref{eq:Etildemto0}) means that 
\[
\mathbb{E}\tilde{m}(z)\rightarrow m_{I}(z).
\]
At last, by Eq. (\ref{eq:mtildem}),   $m_{I}(z)$ is the limit of $\mathbb{E}m_{A}(z)$. 
\end{proof}

\section{Random Inner-product Kernel Matrices}\label{sec:L2f}

\subsection{Model and Notations}\label{subsec:modelnotation}

Let $X_{1},\cdots,X_{n}$ be i.i.d random vectors in $\mathbb{R}^{p}$ 
and assume that $X_{i}\sim \mathcal{N}(0,p^{-1}I_{p})$. 
The random kernel  matrix $A$ is defined in Eqn. (\ref{eq:kernalmatrixA}) with the kernel function $f(\xi; p)$, and we define
\begin{equation}
k(x;p)=\sqrt{p}f(\frac{x}{\sqrt{p}};p).
\label{eq:definek(x)}
\end{equation}
In many cases of interest $f(\xi;p)$ does not depend on $p$, or the dependency is in the form of some rescaling or normalization. However, we formulate our result in a general form, keeping the dependency of $k(x;p)$ on $p$, and require $k(x;p)$ to satisfy certain conditions. We will see that those conditions are often satisfied in the cases of interest (Remark \ref{rk:k(x;p)=k(x)} and Remark \ref{rk:f(xi;p)=f(xi)}).

Let $X$ and $Y$ be two independent random vectors distributed as $\mathcal{N}(0,p^{-1}I_{p})$, and define $\xi_p = \sqrt{p} X^TY $. Denote the probability density of $\xi_p$ by $q_p(x)$, and the $L^2$ spaces  ${\cal H}_{p}=L^{2}(\mathbb{R},q_{p}(x)dx)$. Let $\{ P_{l,p}(x), l=0,1,\cdots \}$ be a set of orthonormal polynomials in ${\cal H}_{p}$, that is \[
\int_{\mathbb{R}} P_{l_1,p}(x)P_{l_2,p}(x)q_p(x) dx  =\delta_{l_1,l_2},
\]
where $\delta_{l,k}$ equals 1 when $l=k$ and 0 otherwise.
We define $P_{l,p}$ $(l\geq 0)$ using the Gram-Schmidt procedure on the monomials $\{1, x, x^2, \ldots\}$, so that $P_{0,p}=1$, $P_{1,p}=x$ (notice that $\mathbb{E}\xi_p^2=1$), and $P_{l,p}$ is a polynomial of degree $l$.  Notice that by the Central Limit Theorem, $\xi_{p} \to {\cal N}(0,1)$ in distribution as $p \to \infty$. We define ${\cal H}_{\mathcal{N}} =L^{2}(\mathbb{R}, q(x)dx)$ where $q(x)=\frac{1}{\sqrt{2\pi}}e^{-x^{2}/2}$. It can be shown (Lemma \ref{lemma:Plpmatchhl}) that for any finite degree $l$, the coefficients of the polynomial $P_{l,p}(x)$ converge to those of the normalized $l$-degree Hermite polynomial, the latter being an orthonormal basis of ${\cal H}_{\mathcal{N}} $. 
 
We formally expand $k(x; p )$ as 
\begin{equation}\label{eq:decompk(x)}
\begin{split}
k(x;p) & =\sum_{l=0}^{\infty} a_{l,p} P_{l,p}(x),\\
a_{l,p} & =\int_{\mathbb{R}} k(x;p) P_{l,p}(x) q_p(x)dx,
\end{split}
\end{equation}
and will later explain how to understand this formal expansion. 
Corresponding to the $l$-th term in Eqn. (\ref{eq:decompk(x)}), we define the random kernel matrix $A_l$ to be 
\begin{equation}
(A_l)_{ij}=\begin{cases}
f_l(X_{i}^{T}X_{j};p), & i\neq j,\\
0, & i=j,
\end{cases}
\label{eq:Aldef}
\end{equation}
where $f_l(\xi;p) = \frac{ a_{l,p} }{ \sqrt{p} } P_{l,p}(\sqrt{p}\xi)$. 

\subsection{Statement of the Main Theorem}\label{subsec:mainthm}

Our main result is stated in Thm. \ref{thm:main}, which establishes the weak convergence of the spectrum of random inner-product kernel matrices. The following conditions are required for $k(x;p)$:
\begin{enumerate}
\item {\bf (C.Variance)} 
For all $p$, $k(x;p) \in {\cal H}_p$, and as $p\to \infty$, $ Var(k(\xi_p;p))= \nu_p \to \nu$ which is a finite non-negative number. We also assume that $a_{0,p} = \mathbb{E}k(\xi_p;p) = 0$ (Remark \ref{rk:a0=0}).

\item {\bf (C.$p$-Uniform)} 
The expansion in Eqn. (\ref{eq:decompk(x)}) converges in $\mathcal{H}_p$  uniformly in $p$. Equivalently, let 
\begin{equation*}
k_{L}(x;p)=\sum_{l = 0}^{L} a_{l,p} P_{l,p}(x),
\end{equation*}
then for any $\epsilon > 0$, there exist $L$ and $p_0$ such that $\sum_{l=L+1}^\infty a_{l,p}^2 < \epsilon$ for $p>p_0$.

\item {\bf (C.$a_1$)}
As $p\to \infty$, $a_{1,p}\rightarrow a$ which is a constant. 
\end{enumerate}

\begin{remark}\label{rk:CVariance}
By condition  {\bf (C.Variance)}, the integrals in Eqn. (\ref{eq:decompk(x)}) are well-defined. The requirement  $\nu_p \to \nu$ can be fulfilled as long as $k(x;p) \in {\cal H}_p$ and is properly scaled. Notice that 
$ \nu_p  = Var(k(\xi_p;p)) = \sum_{l=1}^\infty a_{l,p}^2$,
thus in condition {\bf (C.$a_1$)}, $a^2 \leq \nu $.
\end{remark}


\begin{remark}\label{rk:k(x;p)=k(x)}
When $k(x;p)=k(x)$, and if 
(1) $k(x) \in {\cal H}_{\mathcal{N}}$, and $\mathbb{E}k(\zeta) = 0$ where $\zeta \sim \mathcal{N}(0,1)$,
and (2) $k(x)$ satisfies 
  \begin{equation}\label{eq:with|qp(x)-q(x)|}
   \int_{\mathbb{R}} k(x)^2 |q_p(x)- q(x)| dx \to 0, \quad p \to \infty,
   \end{equation}
then the three conditions are satisfied and  $\nu_p  \to  \nu_\mathcal{N} :=\mathbb{E}k(\zeta)^2 $, and $a_{1,p}\to a_\mathcal{N} := \mathbb{E}\zeta k(\zeta)$  (Lemma \ref{lemma:k(x;p)=k(x)}).
Eqn. (\ref{eq:with|qp(x)-q(x)|}) holds as long as the singularity in the integral, say at $x =\infty$ or $k(x) = \infty$, can be controlled $p$-uniformly. This is the case, for example, when $k(x)$ is bounded, or when $k(x)$ is bounded on $|x|\leq R$ for any $R>0$ and $k(x)^2$ is $p$-uniformly integrable at $x \to \infty$ (Lemma \ref{lemma:with|qp-q|}). It is also possible for  $k(x)$ to be unbounded. See Sec. \ref{subsec:numerical} for an example of $k(x)$ that diverges at $x=0$. 
\end{remark}

\begin{remark}\label{rk:f(xi;p)=f(xi)}
When $f(\xi,p) = f(\xi)$, the three conditions generally need to be checked for $k(x;p)$ case by case. For the special situation where $f(\xi)$ is $C^1$ at $\xi=0$, see Remark \ref{rk:smoothf}.
\end{remark}

\begin{theorem}[the limiting spectrum of random inner-product kernel matrices]\label{thm:main}
 Suppose that $X_1, \cdots, X_n \sim \mathcal{N}(0, p^{-1}I_p)$ are i.i.d., and $k(x;p)$ satisfies conditions {\bf (C.Variance)}, {\bf (C.$p$-Uniform)} and {\bf (C.$a_1$)}. Then, as $p,n\rightarrow\infty$ with $p/n=\gamma$, $ESD_{A}$ (the empirical spectral density of the random kernel matrix $A$, defined in Eqn. (\ref{eq:ESD})) converges weakly to a continuous probability measure on $\mathbb{R}$ in the almost sure sense. The Stieltjes transform of the limiting spectral density is the solution of the following algebraic equation
\begin{equation}
-\frac{1}{m(z)}=z + a \left(1-\frac{1}{1+\frac{a}{\gamma}m(z)}\right)+\frac{ \nu -a^2}{\gamma}m(z),
\label{eq:mzeqn}
\end{equation}
which is at most cubic, and involves three parameters: $\nu $ (defined in {\bf (C.Variance)}), $a$  (defined in {\bf (C.$a_1$)}) and $\gamma$. Eqn. (\ref{eq:mzeqn}) has a unique solution $m(z)$ with positive imaginary part (Lemma \ref{lemma:unique}), and the explicit formula of
\begin{equation}\label{eq:y(u)def}
y(u):=\lim_{v\to 0+} \Im(m(u+iv))
\end{equation} 
is given in Appendix \ref{sec:mzeqnsolution}.
\end{theorem}

\begin{remark} \label{rk:a0=0}
We assume $a_{0,p}=0$, since otherwise it results in adding to the kernel matrix a perturbation $\frac{1}{\sqrt{p}}a_{0,p}(\mathbf{1}_n\mathbf{1}_{n}^{T}-I_n)$, where $\mathbf{1}_{n}$ is the all-ones vector of length $n$ and $I_n$ is the identity matrix. The limiting spectral density of a sequence of Hermitian matrices with growing size ($n \rightarrow \infty$) is invariant to a  finite-rank perturbation (with rank that does not depend on $n$), see Thm. A.43 in \cite{bai}.
\end{remark}

\begin{remark}
Recall the definition of $A_l$ in Eqn. (\ref{eq:Aldef}). The limiting spectral density of $A_1$ is the M.P. distribution. For this case, $f(\xi;p)= a \xi$, or equivalently $k(x;p)= a x$, for some constant $a$. Then, the expansion in Eqn. (\ref{eq:decompk(x)}) has one term, $a_{1,p}= a$, $\nu_p = a^2$, and Eqn. (\ref{eq:mzeqn}) is reduced to Eqn. (\ref{eq:mz1}). 
\end{remark}
\begin{remark}\label{rk:Al2+}
The limiting spectral density of $A_{l}$ $(l\geq2$) is a semi-circle. Moreover, the limiting density of any partial sum (finite or infinite) of $A_2,A_3,\cdots$ is a semi-circle, whose squared radius equals the sum of the squared radii of the semi-circle of each $A_{l}$.
\end{remark}
\begin{remark}\label{rk:smoothf}
For random kernel matrices with locally smooth kernel functions, the limiting spectral density is the M.P. distribution. Specifically, if $f(\xi;p)=f(\xi)$, and is locally $C^1$ at $\xi=0$, one can show (Lemma \ref{lemma:fsmooth}) that the result in the theorem holds and $a^2 = \nu = (f'(0))^2$. In other words, the linear term in Eqn. (\ref{eq:decompk(x)}) determines the limiting spectral density, in agreement with the result in \cite{karoui10}.
\end{remark}

The proof of Thm. \ref{thm:main} is given in Section \ref{sec:proof}. Before presenting the proof, we analyze some examples of kernel functions numerically.

\subsection{Numerical Experiments}\label{subsec:numerical}

We compare the eigenvalue histogram and the theoretical limiting spectral density numerically. In the subsequent figures, the eigenvalues that produce the empirical histogram are computed by MATLAB's {\textsf eig} function and correspond to a single realization of the random kernel matrix. The ``theoretical curve" is calculated using the ``inversion formula" Eqn. (\ref{eq:inversionformula}) and Eqn. (\ref{eq:y(u)}), which is the expression for $y(u; a, \nu, \gamma)$ defined in Eqn. (\ref{eq:y(u)def}).

\subsubsection{Example: $ Sign(x) $}

\begin{figure}
\centering
\includegraphics[width = 0.45\linewidth]{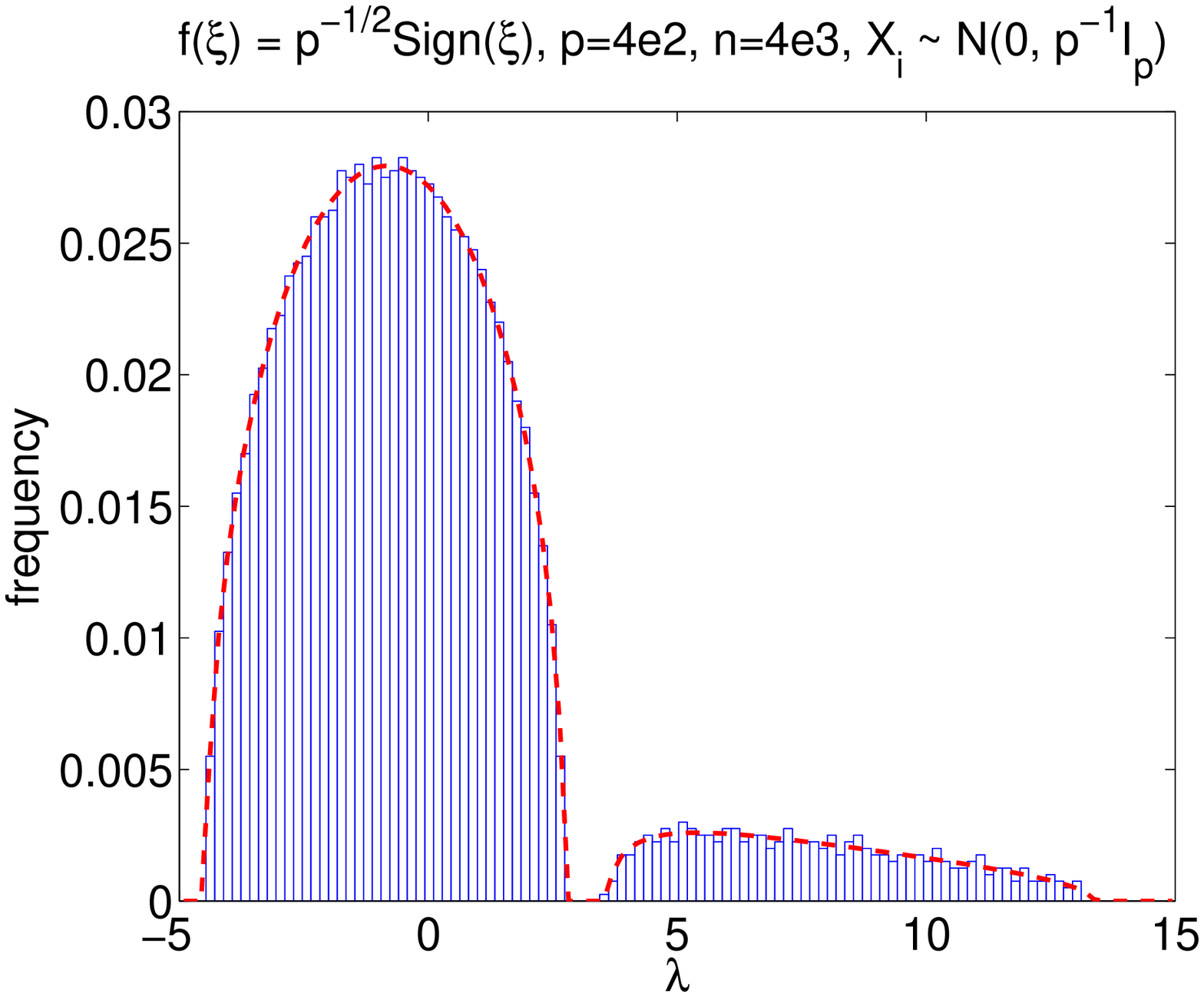}
\includegraphics[width = 0.45\linewidth]{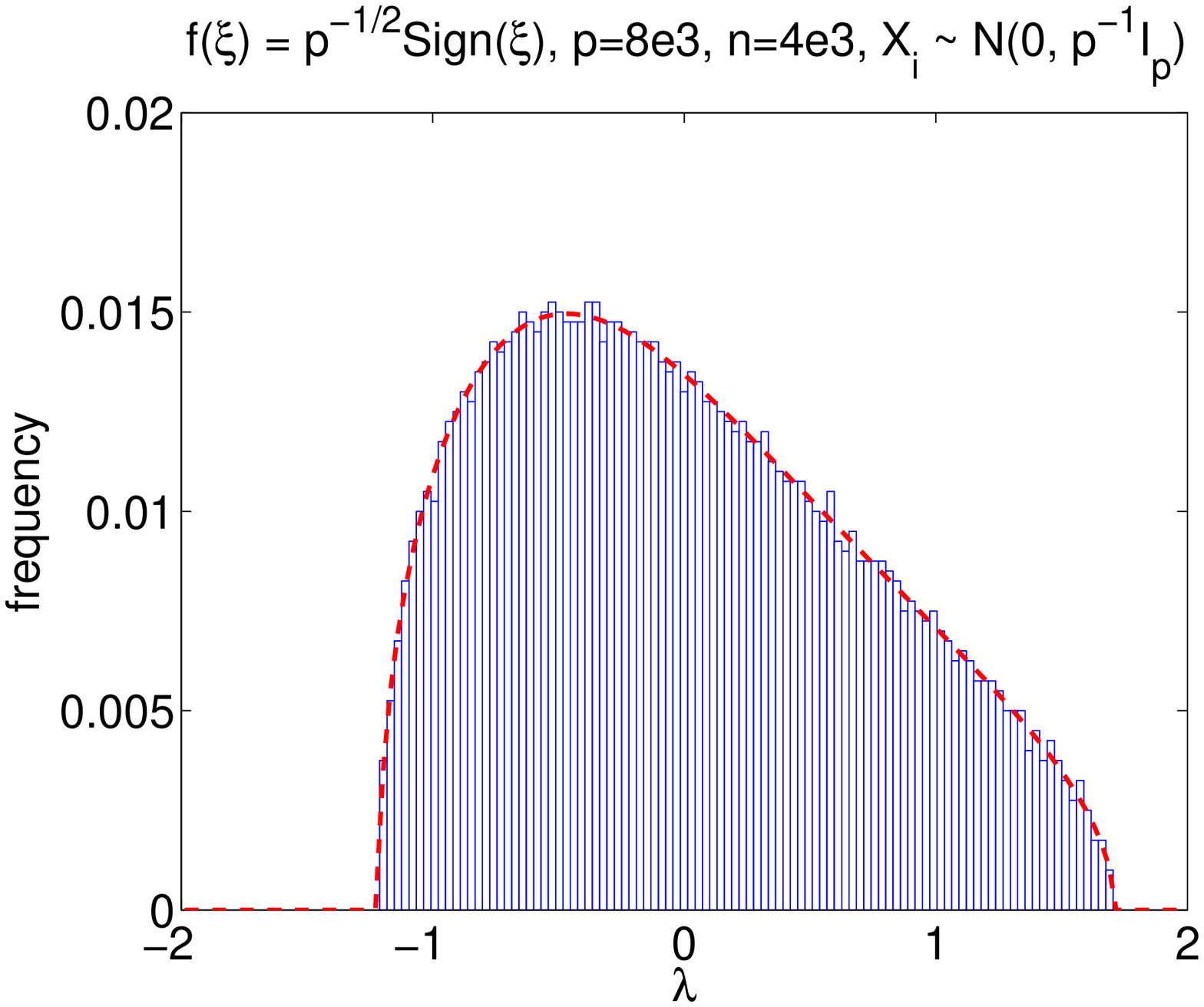}
\caption{\label{fig:signkernel_XiN}
Random kernel matrix with the Sign kernel,
and $X_i \sim \mathcal{N}(0, p^{-1}I_p)$.
(Left) $p = 4\times10^2$, $n = 4\times10^3$, $\gamma = p/n = 0.1$.
(Right) $p = 8\times 10^3$, $n =  4\times 10^3$, $\gamma = p/n = 2$.
The blue-boundary bars are the empirical eigenvalue histograms, and the red broken-line curves are the theoretical prediction of the eigenvalue densities by Thm. \ref{thm:main}. 
}
\end{figure}

As an example of a discontinuous kernel function, let
\[
k(x;p) = k(x) = Sign (x),
\]
where $Sign(x)$ is 1 when $x > 0 $ and -1 otherwise. Since $|k(x)| =1 $, $k(x)$ is bounded, and according to Remark \ref{rk:k(x;p)=k(x)}, by Lemma \ref{lemma:k(x;p)=k(x)} and Lemma \ref{lemma:with|qp-q|}, $k(x)$ satisfies conditions {\bf (C.Variance)}, {\bf (C.$p$-Uniform)} and {\bf (C.$a_1$)}. Meanwhile, $a = \mathbb{E}|\zeta|= \sqrt{2/\pi}$, and $\nu_p =1$ for all $p$, thus $\nu = 1$.

Fig. \ref{fig:signkernel_XiN} is for $X_i \sim \mathcal{N}(0, p^{-1}I_p)$. Notice that for the sign kernel, the two models $X_i \sim \mathcal{N}(0, p^{-1}I_p)$ and $X_i \sim \mathcal{U}(S^{p-1})$ result in the same probability law of the random kernel matrix. This is due to the fact that $Sign(X_i^TX_j) = Sign((X_i/|X_i|)^T(X_j/|X_j|))$ and that  if $X_i \sim \mathcal{N}(0, p^{-1}I_p)$ then $X_i/|X_i| \sim \mathcal{U}(S^{p-1})$. As such, the results for $X_i \sim \mathcal{U}(S^{p-1})$ are omitted.

The following serves as a motivation for the sign kernel matrix. Consider a network of $n$ ``subjects" represented by $X_1, \dots, X_n$ lying in $\mathbb{R}^p$. Subjects $i$ and $j$ have a friendship relationship if they are positively correlated, i.e., if $X_i^TX_j > 0$, and a non-friendship relationship if $X_i^TX_j < 0$. The off-diagonal entries of the  $n$-by-$n$ kernel matrix $A$ are all $\pm 1$ representing the friendship/non-friendship relationships. This model has the merit that if $i$ and $j$ are friends, and $j$ and $k$ are also friends, then chances are greater that $i$ and $k$ are also friends. When the $X_i$'s are i.i.d uniformly distributed on the unit sphere in $\mathbb{R}^p$ and $p$ is fixed, according to \cite{koltchinskii2000random}, as $n$ grows to infinity the top $p$ eigenvectors of the kernel matrix $A$ converge, up to a multiplying constant and a global rotation, to the coordinates of the $n$ data points. In this case, the eigen-decomposition of the sign kernel matrix recovers the positioning of the subjects in the whole community from their pairwise relationships. On the other hand, Thm. \ref{thm:main} covers the more realistic case of the ``large $p$, large $n$" regime.

\subsubsection{Example: $|x|^{-r} ~ (r < 1/2)$}

\begin{figure}
\centering
\includegraphics[width = 0.45\linewidth]{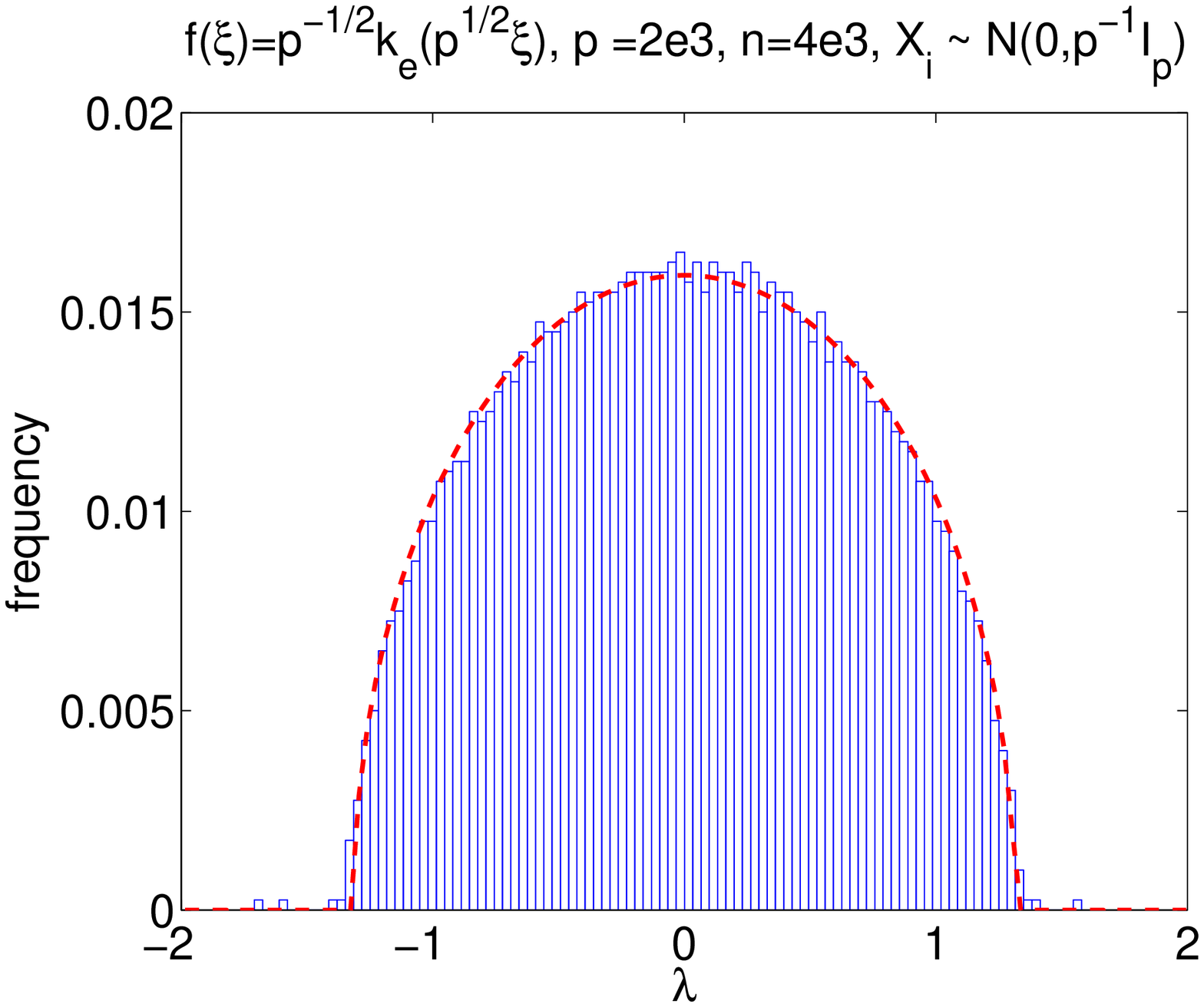}
\includegraphics[width = 0.45\linewidth]{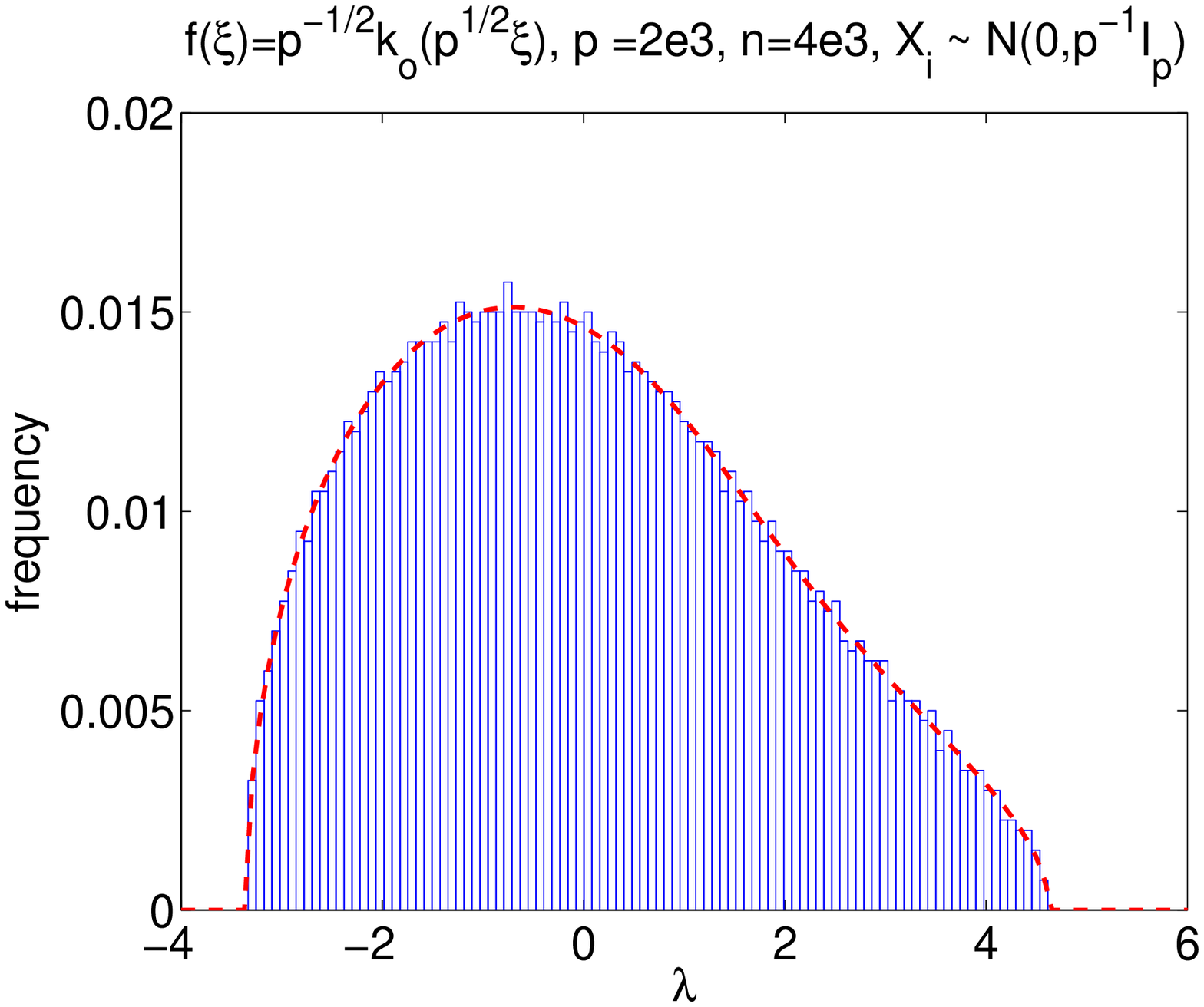}
\caption{\label{fig:divergekernel_XiN}
Random kernel matrix where $k(x) = k_e(x) = |x|^{-1/4} - \mathbb{E}|\zeta|^{-1/4}$ (left) and $k_o(x) = Sign(x)|x|^{-1/4}$ (right). $X_i \sim \mathcal{N}(0, p^{-1}I_p)$, and $p = 2\times10^3$, $n = 4\times10^3$, $\gamma = p/n = 0.5$. 
}
\end{figure}

As examples of unbounded kernel functions, we study the even function \[
k_e(x) = |x|^{-r} - \mathbb{E}|\zeta|^{-r} 
\] and the odd function \[
k_o(x) = Sign(x)|x|^{-r}, 
\] where $r < 1/2$ so as to guarantee the integrability of $k(x)^2$ at $x =0 $. 

Notice that for both cases, $|k(x)|$ is bounded on $\{ |x|> R\}$ for any $R > 0$, and diverge at $x =0$. Meanwhile, $k(x)^2 = |x|^{-2r}$ is integrable at $x=0$, and with the fact that $q_p(x) \leq q_p(0) \to q(0) = 1/\sqrt{2\pi}$, Eqn. (\ref{eq:with|qp(x)-q(x)|}) still holds.  Thus by Lemma \ref{lemma:k(x;p)=k(x)}, Thm. \ref{thm:main} applies to both $k_e$ and $k_o$. By
\[
\mathbb{E}|\zeta|^{-r}  = \sqrt{\frac{2}{\pi}}2^{-(r+1)/2}\Gamma(\frac{1-r}{2})
\] 
where $\Gamma(\cdot)$ is the Gamma function, and similarly for $\mathbb{E}|\zeta|^{-2r}$, the constants $\nu$ and $a$ for both $k_e$ and $k_o$ can be explicitly computed. For $k_e$, $\nu = Var(|\zeta|^{-r})$ and $a = 0$. For $k_o$, $\nu = |\zeta|^{-2r}$, and
\[
a  = \mathbb{E}|\zeta|^{1-r} = \sqrt{\frac{2}{\pi}}2^{-r/2}\Gamma(1-\frac{r}{2}). 
\]

The numerical results for $r=1/4$ with $X_i \sim \mathcal{N}(0, p^{-1}I_p)$ are shown in Fig. \ref{fig:divergekernel_XiN}. The empirical histograms  for $X_i \sim \mathcal{U}(S^{p-1})$ look almost identical and are therefore omitted. In the left panel of Fig. \ref{fig:divergekernel_XiN}, the empirical spectral density is close to a semi-circle, as our theory predicts. Notice that for $ r = 1/4$, the off-diagonal entries of the random kernel matrix do not have a 4th moment. However this does not contradict the ``Four Moment Theorem" for random matrices with i.i.d entries \cite{tao2011random} since the model for random kernel matrices is different.  

\section{Proof of the Main Theorem}\label{sec:proof}

The model and the notations are the same as in Sec. \ref{subsec:modelnotation}. The proof of Thm. \ref{thm:main} is provided in Sec. \ref{subsec:bigproof}. Prior to the proof, in Sec. \ref{subsec:hermite} we review some useful properties of Hermite polynomials, and in Sec. \ref{subsec:boundoftopeig} we introduce an asymptotic upper bound for the expected value of the spectral norm of random kernel matrices. The other model $X_{i}\sim\mathcal{U}(S^{p-1})$ is analyzed in Sec. \ref{subsec:XiS}, where it is shown that the result of Thm. \ref{thm:main} still holds.

\subsection{Orthonormal Polynomials}\label{subsec:hermite}

\subsubsection{${\cal H}_{\mathcal{N}}$ and normalized Hermite polynomials}
Define the normalized Hermite polynomials as 
\begin{equation}\label{eq:hl(x)}
h_{l}(x)=\frac{1}{\sqrt{l!}}H_{l}(x), \quad l=0,1,\cdots
\end{equation}
where $H_l(x)$ is the $l$-degree Hermite polynomial, satisfying
\[
 \int_{\mathbb{R}}H_{l_1}(x) H_{l_2}(x)q(x)dx =\delta_{l_1,l_2}\cdot l_1!.
\]
Thus, $\{h_l(x), l = 0,1,\cdots\}$ form an orthonormal basis of ${\cal H}_{\mathcal{N}}$.  The explicit formula of $H_l$ is \cite{abramowitz1964handbook}
\[
H_{l}(x)=l!\sum_{k=0}^{\lfloor l/2\rfloor}(-\frac{1}{2})^{k}\frac{1}{k!(l-2k)!}x^{l-2k}.
\] Also, the derivative of $H_{l}(x)$ satisfies the recurrence relation $ {H_{l}}' (x)=lH_{l-1}(x)$ for  $l \geq 1$, and as a result, 
\begin{equation}
{h_{l}}'(x)=\sqrt{l}h_{l-1}(x).
\label{eq:hermiteprime_h}
\end{equation}

\subsubsection{${\cal H}_p$ and $P_{l,p}(x)$}

Recall that the random variable $\xi_p$ converges in distribution to $\mathcal{N}(0,1)$ as $p \to \infty$. Meanwhile, the moments of $\xi_p$ approximate those of $\mathcal{N}(0,1)$: 
\begin{equation}
\mathbb{E}\xi_p^{k}=\begin{cases}
(k-1)!!+{\cal O}_{k}(1)p^{-1}, & k\:\text{even;}\\
0, & k\:\text{odd}.
\end{cases}\label{eq:Exi^k}
\end{equation} 
Eq. (\ref{eq:Exi^k}) is verified by directly computing the moments of $\xi_p$ using the model, i.e. $\xi_p = \sqrt{p} X^TY $ and  $X$ and $Y$ are independently distributed as $\mathcal{N}(0,p^{-1}I_{p})$. With the following lemma, Eq. (\ref{eq:Exi^k}) implies the asymptotic consistency between $P_{l,p}$ and $h_l$.

\begin{lemma}[convergence of $P_{l,p}$ to $h_l$]\label{lemma:Plpmatchhl}
Let $\{P_{l,p}, l=0,1,\cdots\}$ be the orthonormal polynomials of $L^2(\mathbb{R}, d\mu_p)$, where $\mu_p$ is a sequence of probability measures. Suppose that the moments of $\mu_p$ approximate those of $\mathcal{N}(0,1)$ in the sense that, for every fixed $k$,
\[
\int_{\mathbb{R}}x^k d\mu_p(x) = \mathbb{E}\zeta^k + \mathcal{O}_k(1)p^{-1}.
\] 
Then, for every fixed degree $l$, \[
P_{l,p}(x) = h_l(x) + \sum_{j=0}^l  (\delta_{l,p})_j x^j,
\]
where $(\delta_{l,p})_j$ satisfy 
\[
\max_{0\leq j \leq l} |(\delta_{l,p})_j| < \mathcal{O}_l(1)p^{-1}.
\]
\end{lemma}

The proof of Lemma \ref{lemma:Plpmatchhl} follows from the fact that the coefficients of the $l$-degree orthogonal polynomials are decided by up to the first $2l$ moments. 

One consequence of Lemma \ref{lemma:Plpmatchhl} is that as $p \to \infty$ 
\begin{equation}\label{eq:boundPlp}
|P_{l,p}(x)|\le{\cal O}_{l}(1)M^{l}, \quad |x| \leq M,
\end{equation}
as the coefficients of $P_{l,p}(x)$ for each $l$ converge to those of $h_l(x)$. Also, Eq. (\ref{eq:hermiteprime_h}) leads to 
\begin{equation}\label{eq:Plpprime}
\begin{split}
P^{'}_{l,p}(x) &= \sqrt{l}P_{l-1,p}(x)+\mathcal{O}_l(1)M^{l-1}p^{-1}, \\
P^{''}_{l,p}(x) &= \sqrt{l(l-1)}P_{l-2,p}(x)+\mathcal{O}_l(1)M^{l-2}p^{-1}, 
\end{split}
\quad l \geq 2.
\end{equation}

Another consequence is the ``asymptotic consistency between the $P_{l,p}$-expansion and the Hermite-expansion" in their first finite-many terms (Lemma \ref{lemma:alptoalN}). This further implies that conditions {\bf (C.Variance)}, {\bf (C.$p$-Uniform)} and {\bf (C.$a_1$)} are satisfied by a large class of kernel functions (Remark \ref{rk:k(x;p)=k(x)}).

\subsection{ Spectral Norm Bound}
\label{subsec:boundoftopeig}

The following lemma gives an upper bound for the expectation of the spectral norm of random kernel matrices whose rescaled kernel function $k(x;p)$ is $P_{l,p}(x)$ (defined in Sec. \ref{subsec:hermite}) for some $l$. The method is by analyzing the 4th moment of the random matrix.

\begin{lemma}[bounding mean spectral norm by the 4th moment]\label{lemma:Es(A)_4thmoment} 
Let $A$ be the random kernel matrix defined in Eq. (\ref{eq:kernalmatrixA}) with the kernel function $f(\xi;p)= p^{-1/2}P_{l,p}(\sqrt{p}\xi)$, $l\ge 1$, where $P_{l,p}$ is defined as in Sec. \ref{subsec:hermite}. $X_i$'s are i.i.d. distributed as $\mathcal{N}(0, p^{-1}I_p)$. Then, as $p,n\rightarrow\infty$, $p/n=\gamma$, 
\[
\mathbb{E}s(A) \leq {\cal O}_{l, \gamma}(1) n^{1/4}.
\]
\end{lemma}

\begin{remark}\label{rk:s(A)isO(1)}
 We are aware of the existence of significant literature on the spectral norm of random matrices. The asymptotic concentration of the largest eigenvalue at its mean value is quantified by the Tracy-Widom Law for Gaussian ensembles (see, e.g. \cite[Chapter 3]{Anderson2010book}) and a large class of Wigner-type matrices (see \cite{soshnikov1999universality}, \cite{tao2010uptotheedge} and references therein), as well as Wishart-type matrices (Remark \ref{rk:A1_2}). For random kernel matrices, $s(A_l)$ is conjectured to be ${\cal O}(1)$, and see more in Sec. \ref{sec:discussion}. However, the bound provided by Lemma \ref{lemma:Es(A)_4thmoment}, though not tight, is sufficient for the proof of our main theorem.  
\end{remark}

\begin{proof}[Proof of Lemma \ref{lemma:Es(A)_4thmoment} ] 
Let $\{\lambda_{i},1\le i\le n\}$ be the eigenvalues of $A$. Since 
\[
s(A)^{4}\le\sum_{i=1}^{n}\lambda_{i}^{4}={\bf Tr}(A^{4})
= \sum_{i,j,k,l}A_{ij}A_{jk}A_{kl}A_{li},
\]
we have
\begin{equation}
\mathbb{E}s(A)\le(\mathbb{E}s(A)^{4})^{1/4}
\le(\sum_{i,j,k,l}\mathbb{E}A_{ij}A_{jk}A_{kl}A_{li})^{1/4}.
\label{eq:Es(A)4thmoment}
\end{equation}
 We observe that for $\mathbb{E}A_{ij}A_{jk}A_{kl}A_{li}$ to be non-zero, in $\{i,j,k,l\}$  neighboring indices must differ since $A_{ii}=0$. Then we have the following cases: 
 \begin{enumerate}

 \item $i=k,j=l$: $\mathbb{E}A_{ij}A_{jk}A_{kl}A_{li} = \mathbb{E}A_{ij}^4 = p^{-2} \mathbb{E} P_{l,p}(\sqrt{p}\xi_{12})^4 = {\cal O}_l(1)p^{-2}$, where the last equality is due to Eqn. (\ref{eq:Exi^k}) and Lemma \ref{lemma:Plpmatchhl}.
 
 \item $i=k,j\ne l$ or $i\ne k,j=l$: $\mathbb{E}A_{ij}A_{jk}A_{kl}A_{li} = p^{-2} (\mathbb{E} P_{l,p}(\sqrt{p}\xi_{12})^2)^2 = (1+{\cal O}_l(1)p^{-1})^2p^{-2} = {\cal O}_l(1)p^{-2}$. 
 
 \item $i\ne k,l\ne j$: when $l = 1$, $\mathbb{E}A_{ij}A_{jk}A_{kl}A_{li} = p^{-3}$. When $l \ge 2$, we have the following estimate (Lemma \ref{lemma:4thmoment_lge2})
\[
 \mathbb{E}A_{ij}A_{jk}A_{kl}A_{li} = {\cal O}_l(1) p^{-4}. 
\]
\end{enumerate} 
Thus, when $l = 1$,
\[
\begin{split}
& \sum_{i,j,k,l}\mathbb{E}A_{ij}A_{jk}A_{kl}A_{li}\\
& \le n^{2}{\cal O}(1)p^{-2}+2n^{3} {\cal O}(1)p^{-2}+n^{4}p^{-3}\\
& = {\cal O}_{\gamma}(1)n+{\cal O}_{\gamma}(1),
\end{split}
\]
and when $l \ge 2$,
\[
\begin{split}
& \sum_{i,j,k,l}\mathbb{E}A_{ij}A_{jk}A_{kl}A_{li}\\
& \le n^{2}{\cal O}_l(1)p^{-2}+2n^{3} {\cal O}_l(1)p^{-2}+n^{4}{\cal O}_l(1)p^{-4}\\
& = {\cal O}_{l,\gamma}(1) n+{\cal O}_{l,\gamma}(1).
\end{split}
\]
Combining the above estimates with Eq. (\ref{eq:Es(A)4thmoment}) leads to the bound wanted.  
\end{proof}

\subsection{Proof of Thm. \ref{thm:main}}\label{subsec:bigproof}

\begin{proof}[Proof of Thm. \ref{thm:main}]

Same as in Sec. \ref{subsec:A1isMP}, it suffices to show the mean convergence of the Stieltjes transform. Specifically, we want to show that for a fixed $z = u+ iv$, $\mathbb{E}m_A(z)$ converges to the unique solution of Eq. (\ref{eq:mzeqn}). Recall that the expansion Eq. (\ref{eq:decompk(x)}) converges $p$-uniformly in ${\cal H}_p$, and we first reduce the general case to that where the expansion has finite many terms. 

{\bf Step 1. Reduction to the case of finite expansion up to order $L$.}

Denote the truncated kernel function up to finite order $L$ by $f_L(\xi;p)= p^{-1/2}k_L(\sqrt{p}\xi;p)$ where (recall that $a_{0,p} =0$ by Remark \ref{rk:a0=0}) 
\[
k_{L}(x;p)=\sum_{l = 1}^{L} a_{l,p} P_{l,p}(x).
\] Let $m_A(z)$ and $m_L(z)$ be the Stieltjes transforms of the random kernel matrix with the kernel function $f(\xi;p)$ and $f_L(\xi;p)$, respectively.   
For a fixed $z$, define 
\begin{equation}\label{eq:RHSdef}
RHS(m;  a ,\nu)=\left(-z-a\left(1-\frac{1}{1+\frac{a}{\gamma}m}\right)-\frac{\nu-a^2}{\gamma}m\right)^{-1}.
\end{equation}
The goal is to show that, as $p, n \to \infty$ with $p/n = \gamma$, $\mathbb{E}m_A$ converges to the solution of Eq. (\ref{eq:mzeqn}) which can be rewritten as $m = RHS(m; a,\nu)$, and it suffices to show that 
\begin{equation}\label{eq:mA-RHS(mA)}
|\mathbb{E}m_A - RHS(\mathbb{E}m_A; a,\nu) | \to 0.
\end{equation}
We need the following lemma, whose proof is left to Appendix \ref{sec:A4}: 

\begin{lemma}[stability of the Stieltjes transform to $L^2$ perturbation in the kernel function]\label{lemma:m_A-m_B} 
Suppose that $X_i$ $(i = 1, \cdots, n)$ are i.i.d random vectors, and the two functions  $f_{A}(\xi;p)$ and  $f_{B}(\xi;p)$ satisfy that with large $p$
\[
\mathbb{E}(f_{A}(X^{T}Y;p)-f_{B}(X^{T}Y;p))^{2}\le\epsilon p^{-1},
\]
where $X$ and $Y$ are two independent random vectors distributed in the same way as $X_i$'s, and $\epsilon$ is some positive constant.
Let $A$ be the $n$-by-$n$ random kernel matrix with the
kernel function $f_{A}(\xi;p)$, and $B$ with $f_{B}(\xi;p)$. Also, let $m_A$ and $m_B$ be the Stieltjes Transforms of $A$ and $B$ respectively. Then  for a fixed $z$,
\[
\mathbb{E}|m_{A}(z)-m_{B}(z)|\le{\cal O}(1)\sqrt{\epsilon}.
\]
\end{lemma}

By condition  {\bf (C.$p$-Uniform)}, for arbitrary $\epsilon>0$, there exists some $L = L(\epsilon)$, so that 
$ \mathbb{E}(k(\xi_p;p)-k_{L(\epsilon)}(\xi_p;p))^{2}\leq \epsilon^{2}$ for all $p$, and then 
\[
\mathbb{E}(f(X^{T}Y;p)-f_{L(\epsilon)}(X^{T}Y;p))^{2}\le\epsilon^{2}p^{-1}.
\]
By Lemma \ref{lemma:m_A-m_B}, 
\[
|\mathbb{E}m_{A}(z) - \mathbb{E}m_{L(\epsilon)}(z)|\le\mathbb{E}|m_{A}(z)-m_{L(\epsilon)}(z)|\le{\cal O}(1)\epsilon.
\]
If in addition we can show that, for any fixed $L$ and some sequence of $a_L(p)$ and $\nu_L(p)$,
\begin{equation}\label{eq:step2goal}
\begin{split}
& |\mathbb{E} m_L - RHS(\mathbb{E}m_L; a_L(p), \nu_L(p)) | \to 0,\\
& a_L(p) \to a, \quad \nu_L(p) \to \nu,
\end{split}
\end{equation}
then Eq. (\ref{eq:mA-RHS(mA)}) holds asymptotically.
 
{\bf Step 2. Convergence of $\mathbb{E}m_{L}(z)$ for finite $L$.}

With slight abuse of notation, we denote the random kernel matrix with kernel function $f_L(\xi;p)$ by $A$. Its Stieltjes transform is denoted by $m_L(z)$. In what follows we sometimes drop the dependence on $p$ and write $f_L(\xi;p)$ as $f_L(\xi)$, and similar for other functions. 

Recall that 
\begin{equation}\label{eq:inv(A-z)nn}
\begin{split}
\mathbb{E}m_L(z) 
& =\mathbb{E}\left((A-zI)^{-1}\right)_{nn}\\
& =\mathbb{E}(-z-A_{\cdot,n}^{T}(A^{(n)}-zI_{n-1})^{-1}A_{\cdot,n})^{-1}.
\end{split}
\end{equation}
 Notations as in Eq. (\ref{eq:Xn_1}, \ref{eq:Xn_2}, \ref{eq:Xn_3}), we have 
 \begin{equation}\label{eq:f(1)f(2)}
\begin{split}
A_{\cdot,n}  & =  f_{(1)}+f_{(2)}, \\
f_{(1)} & := a_{1,p}|X_{n}|\eta, \\
f_{(2)} & := (f_{>1}(\xi_{1n}),\cdots,f_{>1}(\xi_{n-1,n}))^{T},
\end{split}
\end{equation}
where $\xi_{in}=|X_{n}|\eta_{i}$ for $1\le i\le n-1$, $\eta : = (\eta_1, \cdots, \eta_{n-1})^T$, and $f_{>1}(\xi) := f_L(\xi) - a_{1,p}\xi$. 
The off-diagonal entries of $A^{(n)}$ are 
\[
A^{(n)}_{ij}=f_L(X_{i}^{T}X_{j})=f_L(\eta_{i}\eta_{j}+\tilde{\xi}_{ij}),\quad1\le i,j\le n-1,i\ne j,
\]
where $\tilde{\xi}_{ij}=\tilde{X}_{i}^{T}\tilde{X}_{j}$.

The typical magnitude of $\eta_{i}$ and $\tilde{\xi}_{ij}$ is $p^{-1/2}$, and specifically, we have the large probability set $\Omega_{\delta}$ defined as
\begin{equation}
\Omega_{\delta}=\{|\eta_{i}|<\delta,|\tilde{\xi}_{ij}|<\delta,\left||X_{n}|^{2}-1\right|<\sqrt{2}\delta,1\le i,j\le n-1,i\neq j\},
\label{eq:def.Omega_delta}
\end{equation}
where $\delta=\frac{M}{\sqrt{p}}$, $M=\sqrt{20\ln p}$.
By Lemma \ref{pf:omega_delta_p-9}, $ \Pr(\Omega_{\delta}^{c})\le{\cal O}(1)p^{-7}$. On $\Omega_{\delta}$, 
\begin{align*}
 f_{L}(\eta_{i}\eta_{j}+\tilde{\xi}_{ij})
  & =a_{1,p} \eta_{i}\eta_{j}+a_{1,p}\tilde{\xi}_{ij} \\
  & \quad + f_{>1}(\tilde{\xi}_{ij})+f_{>1}^{'}(\tilde{\xi}_{ij})\eta_{i}\eta_{j}+t_{ij},
\end{align*}
where 
\[
t_{ij}=\frac{1}{2}f_{>1}^{''}(\theta_{ij})(\eta_{i}\eta_{j})^{2}.
\]
Recall that $f_{>1}(\xi) = \frac{1}{\sqrt{p}}\sum_{l=2}^L a_{l,p}P_{l,p}(\sqrt{p}\xi)$, and by Eq. (\ref{eq:Plpprime}), 
\begin{equation}\label{eq:f(2)prime}
f_{>1}^{'}(\xi) = \sum_{l=2}^L a_{l,p}(\sqrt{l}P_{l-1,p}(\sqrt{p}\xi) + \mathcal{O}_l(1)M^{l-1}p^{-1}),
\end{equation}
and 
\begin{equation}\label{eq:f(2)doubleprime}
f_{>1}^{''}(\xi) = \sqrt{p}\sum_{l=2}^L a_{l,p}(\sqrt{l(l-1)}P_{l-2,p}(\sqrt{p}\xi) + \mathcal{O}_l(1)M^{l-2}p^{-1}).
\end{equation}
We define 
\[
\begin{split}
\tilde{A}_{ij}^{(n)} & =a_{1,p}\tilde{\xi}_{ij}+f_{>1}(\tilde{\xi}_{ij}), \\
\tilde{F}_{ij} & = \frac{1}{\sqrt{p}} \sum_{l=2}^L a_{l,p} \sqrt{l}P_{l-1,p}(\sqrt{p}\tilde{\xi}_{ij}),
\end{split}
\quad i\neq j,
\]
and set the diagonal entries to be zeros for both $\tilde{A}^{(n)}$ and $\tilde{F}$, then
\[
A^{(n)}=\tilde{A}^{(n)}+ a_{1,p}(\eta\eta^{T}-D_{\eta})+\sqrt{p}W\tilde{F}W + T,
\]
where $T$ is Hermitian with $T_{ij}=t_{ij} + \frac{1}{\sqrt{p}} f_{>1}^{'}(\tilde{\xi}_{ij}) - \tilde{F}_{ij}$, and $W= \text{diag}\{\eta_{1},\cdots,\eta_{n-1}\}$. 
We have (recall that $\sum_{l=1}^L a_{l,p}^2$ is bounded by some $\mathcal{O}(1)$ constant for all $p $, by Remark \ref{rk:CVariance})
\begin{enumerate}
\item Since $\theta_{ij}$ is between $\tilde{\xi}_{ij}$ and $\tilde{\xi}_{ij}+\eta_{i}\eta_{j}$, and both $\tilde{\xi}_{ij}$ and $\eta_{i}$ are bounded in magnitude by $\delta=p^{-1/2}M$, then $|\theta_{ij}|\leq\delta+\delta^{2}\leq1.01\delta=p^{-1/2}1.01M$. Thus, by Eq. (\ref{eq:f(2)doubleprime}, \ref{eq:boundPlp}), $|f_{>1}^{''}(\theta_{ij})|\leq \sqrt{p}\mathcal{O}_L(1)M^{L-2}$, and then $|t_{ij}|\le{\cal O}_{L}(1)M^{L+2}p^{-3/2}$. Together with Eq. (\ref{eq:f(2)prime}), 
\begin{align}
|T_{ij}|
& \le {\cal O}_{L}(1)M^{L+2}p^{-3/2} + {\cal O}_{L}(1)M^{L-1}p^{-3/2} \nonumber \\
& = {\cal O}_{L}(1)M^{L+2}p^{-3/2}. 
\label{eq:TaylorBoundT}
\end{align}
As a result,
\begin{align}
s(T-a_{1,p}D_{\eta})\cdot\mathbf{1}_{\Omega_{\delta}} 
& \le s(T)\cdot\mathbf{1}_{\Omega_{\delta}}  + |a_{1,p}|\delta \nonumber \\
& = {\cal O}_{L}(1)M^{L+2}p^{-1/2}+{\cal O}(1)Mp^{-1/2} \nonumber \\
& ={\cal O}_{L}(1)M^{L+2}p^{-1/2}. \label{eq:s(T-aD)}
\end{align}

\item 
$\tilde{F} $ can be written as $\sum_{l=1}^{L-1} a_{l,p}\sqrt{l}\tilde{F}_l$, where Lemma \ref{lemma:Es(A)_4thmoment} applies to each $\tilde{F}_l$, and the coefficients $a_{l,p}$ for $1 \leq l \leq L-1$ are uniformly bounded by some constant since $\sum_{l=1}^L a_{l,p}^2 = \nu_p \to \nu$ by Condition {\bf (C.Variance)}. Thus we have
\begin{equation}
\mathbb{E}s(\tilde{F})\le \sum_{l=2}^L \sqrt{l} {\cal O}_l(1) p^{1/4}= {\cal O}_{L}(1)p^{1/4}, 
\label{eq:TaylorBoundtildeF}
\end{equation} 
and as a result,
\begin{equation}
\mathbb{E}s(\sqrt{p}W\tilde{F}W)\cdot\mathbf{1}_{\Omega_{\delta}}\le M^{2}p^{-1/2}\mathbb{E}s(\tilde{F})\leq{\cal O}_{L}(1)M^{2}p^{-1/4}.
\label{eq:Es(sqrtpWtildeFW)}
\end{equation}
 \end{enumerate}
 
Now we break the quantity $A_{\cdot,n}^{T}(A^{(n)}-zI_{n-1})^{-1}A_{\cdot,n}$ into the following pieces: define $\hat{A}^{(n)}=a_{1,p}\eta\eta^{T}+\tilde{A}^{(n)}$, and recall that $A_{\cdot,n}= f_{(1)}+f_{(2)}$ as defined in Eq. (\ref{eq:f(1)f(2)}), 
\begin{align}
   A_{\cdot, n}^{T}(A^{(n)}-zI_{n-1})^{-1}A_{\cdot, n} 
= & A_{\cdot, n}^{T}(\hat{A}^{(n)}-zI_{n-1})^{-1}A_{\cdot, n}-A_{\cdot, n}^{T}(A^{(n)}-zI_{n-1})^{-1} \nonumber \\
 & \cdot(\sqrt{p}W\tilde{F}W+T -a_{1,p}D_{\eta})(\hat{A}^{(n)}-zI_{n-1})^{-1}A_{\cdot, n} \nonumber \\
= & f_{(1)}^{T}(\hat{A}^{(n)}-zI_{n-1})^{-1}f_{(1)} \nonumber \\ 
  &  +f_{(2)}^{T}(\hat{A}^{(n)}-zI_{n-1})^{-1}f_{(2)}+r_{2}-r_{1} 
 \label{eq:step2_(1)(2)r2r1}
\end{align}
where 
\begin{equation}\label{eq:step2_r1r2}
\begin{split}
r_{2} & =2f_{(1)}^{T}(\hat{A}^{(n)}-zI_{n-1})^{-1}f_{(2)},\\
r_{1} & =A_{\cdot, n}^{T}(A^{(n)}-zI_{n-1})^{-1}(\sqrt{p}W\tilde{F}W+T-a_{1,p}D_{\eta}) \\
      & \quad \cdot(\hat{A}^{(n)}-zI_{n-1})^{-1}A_{\cdot, n}.
\end{split}
\end{equation}
For $r_2$,
\begin{align}
r_{2} 
 & =2a_{1,p}|X_{n}|\eta^{T}(\hat{A}^{(n)}-zI_{n-1})^{-1}f_{(2)} \nonumber\\
 & =2a_{1,p}f_{(2)}^{T}(\hat{A}^{(n)}-zI_{n-1})^{-1}(|X_{n}|\eta)\nonumber \\
 & =2a_{1,p} \{f_{(2)}^{T}(\tilde{A}^{(n)}-zI_{n-1})^{-1}(|X_{n}|\eta) \nonumber\\
 & \quad - f_{(2)}^{T}(\tilde{A}^{(n)}-zI_{n-1})^{-1}a_{1,p}\eta\eta^{T}(\hat{A}^{(n)}-zI_{n-1})^{-1}(|X_{n}|\eta) \}\nonumber \\
 & :=2a_{1,p}(r_{2,1}-r_{2,2}),\label{eq:step_2r21r22}
\end{align}
and by moment method we can show that (Lemma \ref{lemma:step2_E|r_2|})
\begin{equation}
\mathbb{E}|r_{2}|\cdot\mathbf{1}_{\Omega_{\delta}}\le{\cal O}_{L}(1)M^{2}p^{-1/2}.
\label{eq:step2_r2}
\end{equation}
To bound $r_{1}$, we restrict ourselves to $\Omega_{\delta}$ where $||A_{\cdot, n}||^{2}=\sum_{i=1}^{n-1}f_{L}(\xi_{in})^{2}\le{\cal O}_{L}(1)M^{L}$,
and with Eq. (\ref{eq:Es(sqrtpWtildeFW)}, \ref{eq:s(T-aD)}) 
\begin{align}
 \mathbb{E}|r_{1}|\cdot\mathbf{1}_{\Omega_{\delta}}
& \le\mathbb{E}(s(\sqrt{p}W\tilde{F}W)+s(T-a_{1,p}D_{\eta}))||A_{\cdot, n}||^{2}\cdot\mathbf{1}_{\Omega_{\delta}} \nonumber \\
 & \le{\cal O}_{L}(1)M^{L}\mathbb{E}(s(\sqrt{p}W\tilde{F}W)+s(T-a_{1,p}D_{\eta})) \nonumber\\
 & ={\cal O}_{L}(1)M^{L}({\cal O}_{L}(1)M^{2}p^{-1/4}+{\cal O}_{L}(1)M^{L+2}p^{-1/2})\nonumber \\
 & ={\cal O}_{L}(1)M^{2L+2}p^{-1/4} \label{eq:step2_r1}.
\end{align}

Furthermore, as in Sec. \ref{subsec:A1isMP}, we can compute the first term in Eq. (\ref{eq:step2_(1)(2)r2r1}): 
\begin{align*}
 f_{(1)}^{T}(\hat{A}^{(n)}-zI_{n-1})^{-1}f_{(1)} 
& =|X_{n}|^{2}a_{1,p}^{2}\eta^{T}(\hat{A}^{(n)}-zI_{n-1})^{-1}\eta\\
 & =|X_{n}|^{2}a_{1,p}\left(1-(1+a_{1,p}\eta^{T}(\tilde{A}^{(n)}-zI_{n-1})^{-1}\eta)^{-1}\right)\\
 & =|X_{n}|^{2}a_{1,p}\left(1-(1+a_{1,p}(\gamma^{-1}\mathbb{E}\tilde{m}(z)+\gamma^{-1}\tilde{r}+r_{(1),2}))^{-1}\right),
\end{align*}
where $\tilde{m}(z) = \frac{1}{n-1}\Tr(\tilde{A}^{(n)}-zI_{n-1})^{-1}$, and
\begin{enumerate}
\item $\tilde{r}=\tilde{m}(z)-\mathbb{E}\tilde{m}(z)$, $\mathbb{E}|\tilde{r}|\le{\cal O}(1)n^{-1/2}$ by Lemma \ref{lemma:m_A-Em_A};
\item The term
\[
r_{(1),2}=\eta^{T}(\tilde{A}^{(n)}-zI_{n-1})^{-1}\eta-\frac{1}{p}\mathbf{Tr}(\tilde{A}^{(n)}-zI_{n-1})^{-1}
\]
is similar to $r_{2}$ in Lemma \ref{lemma:E|m_A-RHS(tildem)|} and satisfies
$ \mathbb{E}|r_{(1),2}|\le{\cal O}(1)p^{-1/2}$. 
\end{enumerate}
Going through a process similar to that in Lemma \ref{lemma:E|m_A-RHS(tildem)|} to bound the denominators, including 
\begin{enumerate}
\item
introducing a large probability set 
\[
\Omega_{(1)}:=\{|\tilde{r}|\le p^{-1/4},|r_{(1),2}|\le p^{-1/4}\},\quad\Pr(\Omega_{(1)}^{c})\le{\cal O}(1)p^{-1/4},
\]
so as to bound $|(1+a_{1,p}\gamma^{-1}\mathbb{E}\tilde{m}(z))^{-1}|$
on $\Omega_{\delta}\cap\Omega_{(1)}$ by ${\cal O}(1)M^{2}$, 
\item 
making use of that $|(1+a_{1,p}\eta^{T}(\tilde{A}^{(n)}-zI_{n-1})^{-1}\eta)^{-1}|$ on $\Omega_{\delta}$ is bounded by ${\cal O}(1)M^{2}$,
\end{enumerate} 
we have 
\begin{equation}
f_{(1)}^{T}(\hat{A}^{(n)}-zI_{n-1})^{-1}f_{(1)}=a_{1,p}\left(1-(1+\frac{a_{1,p}}{\gamma}\mathbb{E}\tilde{m}(z))^{-1}\right)+r_{(1)},
\label{eq:step2_(1)_1}
\end{equation}
where 
\begin{equation}
\mathbb{E}|r_{(1)}|\cdot\mathbf{1}_{\Omega_{\delta}\cap\Omega_{(1)}}\le{\cal O}(1)M^{4}p^{-1/2}.
\label{eq:step2_(1)_2}
\end{equation}

We turn to compute the second term in Eq. (\ref{eq:step2_(1)(2)r2r1}). We have
\begin{align}
 f_{(2)}^{T}(\hat{A}^{(n)}-zI_{n-1})^{-1}f_{(2)} 
& =f_{(2)}^{T}(\tilde{A}^{(n)}-zI_{n-1})^{-1}f_{(2)}^{T}\nonumber \\
 & \quad - f_{(2)}^{T}(\tilde{A}^{(n)}-zI_{n-1})^{-1}a_{1,p}\eta\eta^{T}(\hat{A}^{(n)}-zI_{n-1})^{-1}f_{(2)}\nonumber \\
 & =\frac{\nu_{>1,p}}{\gamma}\mathbb{E}\tilde{m}(z)+\frac{\nu_{>1,p}}{\gamma}\tilde{r}+r_{(2),2}-r_{(2),3}\label{eq:step2_(2)}
\end{align}
where \[
\nu_{>1,p} = \mathbb{E}(f_{(2)})_i^2 = \mathbb{E} f_{>1}(\xi_{in})^2 = \nu_p - a_{1,p}^2,
\] and 
\begin{align*}
r_{(2),2} & =f_{(2)}^{T}(\tilde{A}^{(n)}-zI_{n-1})^{-1}f_{(2)}^{T}-\frac{\nu_{>1,p}}{p}\mathbf{Tr}(\tilde{A}^{(n)}-zI_{n-1})^{-1},\\
r_{(2),3} & =f_{(2)}^{T}(\tilde{A}^{(n)}-zI_{n-1})^{-1}a_{1,p}\eta\eta^{T}(\hat{A}^{(n)}-zI_{n-1})^{-1}f_{(2)}\\
 & =a_{1,p}(\eta^{T}(\hat{A}^{(n)}-zI_{n-1})^{-1}f_{(2)})r_{2,1}.
\end{align*}
For $r_{(2),2}$, by a moment method argument similar to the first part in the proof of Lemma \ref{lemma:step2_E|r_2|}, we have
\begin{equation}
\mathbb{E}|r_{(2),2}|\le{\cal O}_{L}(1)p^{-1/2}.
\label{eq:step2_(2)_(2)2}
\end{equation}
To bound $r_{(2),3}$, we restrict ourselves to $\Omega_{\delta}$, where 
\[
|f_{(2)}(\xi_{in})|\le{\cal O}_{L}(1)M^{L}p^{-1/2},\quad|\eta_{i}|\le Mp^{-1/2},\quad1\le i\le n-1,
\]
thus 
\begin{align*}
  |a_{1,p}\eta^{T}(\hat{A}^{(n)}-zI_{n-1})^{-1}f_{(2)}|\cdot\mathbf{1}_{\Omega_{\delta}} 
  & \le {\cal O}(1) s((\hat{A}^{(n)}-zI_{n-1})^{-1})||\eta||\cdot||f_{(2)}||\\
 & \le\frac{ {\cal O}(1)}{v}\sqrt{{\cal O}(1)M^{2}}\sqrt{{\cal O}_{L}(1)M^{2L}}={\cal O}_{L}(1)M^{L+1},
\end{align*}
and then 
\begin{align}
\mathbb{E}|r_{(2),3}|\cdot\mathbf{1}_{\Omega_{\delta}} & =\mathbb{E}|r_{2,1}||a_{1,p}(\eta^{T}(\hat{A}^{(n)}-zI_{n-1})^{-1}f_{(2)})|\cdot\mathbf{1}_{\Omega_{\delta}}\nonumber \\
 & \le{\cal O}_{L}(1)M^{L+1}\mathbb{E}|r_{2,1}|\nonumber \\
 & \le{\cal O}_{L}(1)M^{L+1}p^{-1/2}.
 \label{eq:step2_(2)_(2)3}
\end{align}
Now puting Eq. (\ref{eq:step2_(1)(2)r2r1},\ref{eq:step2_r1},\ref{eq:step2_r2},\ref{eq:step2_(1)_1},\ref{eq:step2_(1)_2},\ref{eq:step2_(2)},\ref{eq:step2_(2)_(2)2},\ref{eq:step2_(2)_(2)3})
together, we have 
\begin{align}
  \mathbb{E} & \left|m_L(z)  -\left(-z-a_{1,p}\left(1-(1+\frac{a_{1,p}}{\gamma}\mathbb{E}\tilde{m}(z))^{-1}\right)-\frac{\nu_{>1,p}}{\gamma}\mathbb{E}\tilde{m}(z)\right)^{-1}\right|\cdot\mathbf{1}_{\Omega_{\delta}\cap\Omega_{(1)}} \nonumber \\
 & \le\frac{2}{v}\mathbb{E}(|r_{1}|+|r_{2}|+|r_{(1)}|+|\nu_{>1,p}\gamma^{-1}\tilde{r}|+|r_{(2),2}|+|r_{(2),3}|)\cdot\mathbf{1}_{\Omega_{\delta}\cap\Omega_{(1)}} \nonumber \\
 & \le{\cal O}_{L}(1)M^{2L+2}p^{-1/4}+{\cal O}_{L}(1)M^{2}p^{-1/2}+{\cal O}(1)M^{4}p^{-1/2} \nonumber \\
 & \quad + {\cal O}(1)n^{-1/2}+{\cal O}_{L}(1)p^{-1/2}+{\cal O}_{L}(1)M^{L+1}p^{-1/2} \nonumber \\
 & ={\cal O}_{L}(1)M^{2L+2}p^{-1/4}\rightarrow 0. \label{eq:limitEmL1}
\end{align}
Meanwhile, similar to the proof of Lemma \ref{bound:E|m_A-m_tildeA|} (making use of the fact that $\mathbb{E}s(\sqrt{p}W\tilde{F}W+T)\cdot\mathbf{1}_{\Omega_{\delta}} \le {\cal O}_L(1)M^2 p^{-1/4}$ and the inequality that $\mathbf{Tr}(AB)\le n\cdot s(A)s(B)$ for $n$-by-$n$ Hermitian matrices $A$ and $B$ ), it can be shown that
\[
\mathbb{E}|m_L(z)-\tilde{m}(z)| \to 0.
\]
With Eq. (\ref{eq:limitEmL1}), we have (dropping the dependence on $z$) 
\[ 
| \mathbb{E}\tilde{m} - RHS(\mathbb{E}\tilde{m}; a_{1,p}, \nu_p) | \to 0,
\] and thus \[
| \mathbb{E}m_L - RHS(\mathbb{E} m_L; a_{1,p}, \nu_p) | \to 0.
\]
At last, by condition {\bf (C.Variance)} and {\bf (C.$a_1$)}, $a_{1,p} \to a$ and $\nu_p \to \nu$. Thus Eq. (\ref{eq:step2goal}) is verified if we set $a_L(p) = a_{1,p}$ and $\nu_L(p) = \nu_p$.
\end{proof}

\subsection{Model $X_{i}\sim\mathcal{U}(S^{p-1})$}\label{subsec:XiS}

We also consider the model where the random vectors $X_i$'s are i.i.d. uniformly distributed on a high-dimensional sphere. For this model, the marginal distribution of the inner-product $\xi_{ij} = X_i^TX_j$ has probability density $Q^{'}_p(u) = A_{p}(1-u^{2})^{(p-3)/2}$, where $A_{p}$ is a normalization constant. Let $\xi^{'}_p $ have the same distribution as $\sqrt{p}\xi_{ij}$, whose probability density is $q^{'}_p(x) = \frac{1}{\sqrt{p}}Q^{'}_p(\frac{x}{\sqrt{p}})$, and let ${\cal H}^{'}_p = L^2(\mathbb{R}, q^{'}_p(x)dx)$. By Lemma \ref{lemma:kmoment_S^p-1},
\[
\mathbb{E}(\xi^{'}_p)^{k}=\begin{cases}
(k-1)!!+{\cal O}_{k}(1)p^{-1}, & k\:\text{even;}\\
0, & k\:\text{odd},
\end{cases}
\]
which echos Eq. (\ref{eq:Exi^k}). As a result, by Lemma \ref{lemma:Plpmatchhl}, the orthonormal polynomials of ${\cal H}^{'}_p$ are asymptotically consistent with the Hermite polynomials. If we expand $k(x;p)$ into the orthonormal polynomials of ${\cal H}^{'}_p$, and require the conditions {\bf (C.Variance)}, {\bf (C.$p$-Uniform)} and {\bf (C.$a_1$)} accordingly, the result in Thm. \ref{thm:main} still holds. 

One way of showing this is sketched as follows: 

Condition on the draw of $X_{n}$, and without loss of generality let $X_{n}=(1,0,\cdots,0)^{T}$. Then 
\[
X_{i}=(u_{i},\sqrt{1-u_{i}^{2}}\tilde{X}_{i}^{T})^{T},\quad1\le i\leq n-1,
\]
where $u_{i}$'s are i.i.d distributed, and $\tilde{X}_{i}$'s are i.i.d. uniformly distributed on the unit sphere in $\mathbb{R}^{p-1}$ independently from $u_{i}$'s. As a result, let $\xi_{ij}=X_{i}^{T}X_{j}$ and $\tilde{\xi}_{ij}=\tilde{X}_{i}^{T}\tilde{X}_{j}$, then
\[
\xi_{ij}=u_{i}u_{j}+\sqrt{1-u_{i}^{2}}\sqrt{1-u_{j}^{2}}\tilde{\xi}_{ij},\quad1\le i,j\le n-1,i\neq j,
\]
which is different from before. However, on the large probability set
\[
\Omega_{\delta}=\{|u_{i}|\le\delta,|\tilde{\xi}_{ij}|\leq\delta,1\le i,j\le n-1,i\neq j,\delta=p^{-1/2}M,M=\sqrt{20\ln p}\},
\] it can be shown that 
\[
\xi_{ij}=u_{i}u_{j}+\tilde{\xi}_{ij}+r_{ij},\quad|r_{ij}|\le\delta^{3}.
\] Thus, the Taylor expansion can be carried out in the same way, where the contribution of the extra $r_{ij}$ term is put into $T_{ij}$ and the bound Eq. (\ref{eq:TaylorBoundT}) remains true.

We still need the mean spectral norm bound to show that Eq. (\ref{eq:TaylorBoundtildeF}) holds, and to use the bound given by the 4th moment (Lemma \ref{lemma:Es(A)_4thmoment} ), it suffices to establish the bound in Lemma \ref{lemma:4thmoment_lge2}. Notice that Gegenbauer polynomials \cite{abramowitz1964handbook} are orthogonal in the space $L^2([-1,1], Q_p^{'}(u)du)$. Gegenbauer polynomials are related to the $p$-spherical harmonics $\{ \phi_j, j\in J \}$, which form an orthonormal basis of $L^{2}(S^{p-1},dP)$. $J = \cup_{l=0}^\infty J_l$, and $\{ \phi_j(X), j\in J_l \}$ are $p$-spherical harmonics of degree $l$, which are homogeneous harmonic polynomials restricted to the surface of the unit sphere. The Gegenbauer polynomial of degree $l$ as a function of $X^TY$, $X,Y \in S^{p-1}$, up to a multiplicative constant, equals 
\[
Z_{l,X}(Y) = \sum_{j\in J_l} \phi_j(X)\phi_j(Y)
\]
which is named ``the $l$-degree zonal harmonic function with axis $X$". We thus define $G_{l,p}(\xi)$ to be 
\begin{equation}\label{eq:Glfromphij} 
G_{l,p}(X^T Y) = \sum_{j\in J_l} \phi_j(X)\phi_j(Y).
\end{equation}
Notice that $G_{1,p}(X^T Y) = pX^TY$, and by convention $G_{0,p} = 1$. $G_{l,p}(\xi)$ is a polynomial of degree $l$ for all $l$, and
\[
\int_{S^{p-1}} \int_{S^{p-1}} G_{l,p}(X^T Y) G_{k,p}(X^T Y)dP(X) dP(Y) = \delta_{l,k}|J_l|.
\]
$|J_{l}|$ is the number of $p$-spherical harmonics of degree $l$, $|J_1| = p$, and for $l \ge 2$
\[
|J_{l}|=\left(\begin{array}{c}
p+l-1\\
l
\end{array}\right)-\left(\begin{array}{c}
p+l-3\\
l-2
\end{array}\right)
= \left( \frac{1}{l!} + \frac{\mathcal{O}_l(1)}{p} \right) p^l.
\]
Thus, the orthonormal polynomials $P_{l,p}(x)$ of the space ${\cal H}^{'}_p$ can be written as
\[
P_{l,p}(x) = \frac{1}{\sqrt{|J_l|}}G_{l,p}(\frac{x}{\sqrt{p}}).
\]
By Eq. (\ref{eq:Glfromphij}), we have\[
\int_{S^{p-1}} G_{l,p}(X_1^T X_2) G_{l,p}(X_2^T X_3) dP(X_2) = G_{l,p}(X_1^T X_3), \quad X_1,X_2,X_3 \in S^{p-1},
\]
which gives that (define $\xi_{ij} = X_i^T X_j$) \[
\mathbb{E} [P_{l,p}(\sqrt{p}\xi_{12}) P_{l,p}(\sqrt{p}\xi_{23}) | X_1, X_3] = \frac{1}{\sqrt{|J_l|}} P_{l,p}(\sqrt{p}\xi_{13}).
\]
As a result, $\mathbb{E}P_{l,p}(\sqrt{p}\xi_{12})P_{l,p}(\sqrt{p}\xi_{23})P_{l,p}(\sqrt{p}\xi_{34})P_{l,p}(\sqrt{p}\xi_{41})$ is bounded by \[
\frac{1}{|J_l|}=\mathcal{O}_{l}(1)p^{-l},
\] which is stronger than the estimate in Lemma \ref{lemma:4thmoment_lge2}. We comment that carrying out this analysis to higher order moments gives a moment-method proof of the convergence to semi-circle law of the ESD of random kernel matrices where $k(x;p) = P_{l,p}(x)$ for $l \ge 2$, under the model $X_{i}\sim\mathcal{U}(S^{p-1})$.

To continue to show the result in Thm. \ref{thm:main}, the mechanism in Sec. \ref{subsec:bigproof} applies to what follows in almost the same way.

Another way of extending to the model where $X_{i}\sim\mathcal{U}(S^{p-1})$ is by comparing to the standard Gaussian case. That is, to replace the $X_i$ by $X_i/|X_i|$ in the model $X_{i}\sim \mathcal{N}(0, p^{-1}I_p)$ and to bound the difference resulted in $m_A(z)$ (reducing to the finite expansion case $k=k_L$ first). This ``comparison" argument can be used to extend the result in Thm. \ref{thm:main} to other models of the distribution of  $X_{i}$'s, but we do not develop this idea any further here.

\section{Summary and Discussion}\label{sec:discussion}

The main theorem, Thm. \ref{thm:main},  establishes the convergence of the spectral density of random kernel matrices in the limit $p,n \to \infty$, $p/n=\gamma$, under the assumption that the random vectors are standard Gaussian. The theorem and the proofs also hold under the condition that $p/n \to \gamma$. Our proof is based on analyzing the Stieltjes transform of the random kernel matrix, and uses the expansion of the kernel function into orthonormal Hermite-like polynomials. The limiting spectral density holds for a larger class of kernel functions than the cases studied in \cite{karoui10}, which are smooth kernels.

 The assumption that the random vectors are standard Gaussian can be weakened. We showed that the result extends to the case that they are uniformly distributed over the unit sphere. Numerical simulations (not reported here) indicate that the limiting spectral density holds for other non-Gaussian random vectors. This includes the case where $X_i$'s are uniformly sampled from the $2^p$ vertices of the hypercube $\{-p^{-1/2},p^{-1/2}\}^p$ (where the value of the sign kernel and the divergent kernel at $x=0$ is set to 0). The universality of the limiting spectral density is however beyond the scope of this paper.


While our paper mainly focused on the limiting spectral density, another question of practical importance concerns the statistics of the largest eigenvalue of random kernel matrices. This include studying the mean, variance, limiting distribution, as well as large deviation bounds for the largest eigenvalue. As discussed in Remark \ref{rk:s(A)isO(1)}, the bound in Lemma \ref{lemma:Es(A)_4thmoment} for the expected value of the spectral norm is far from being sharp. Numerical simulations (not reported here) have shown that for the models studied in this paper, the largest/smallest eigenvalue lies at the right/left end of the support of the limiting spectral density, and thus both of them are conjectured to be $\mathcal{O}(1)$ almost surely. We are not aware of any result concerning the limiting probability law of the largest eigenvalue of random kernel matrices, except for the one in \cite{karoui10} where the kernel function is assumed to have strong ($C^3$) regularity. 
Many other interesting questions can be asked from the RMT point of view, e.g. the ``eigenvalue spacing" problem, namely the ``local law" of eigenvalues. If the asymptotic concentration of the eigenvalues at the ``local level" could be established, one consequence would be that the top eigenvalue can be shown to concentrate at the right end of the limiting spectral density.

There are several interesting extensions of the inner-product kernel matrix model. The first possible extension is to distance kernel functions of the form $f(X_i, X_j) = f(|X_i - X_j|)$, which are popular in machine learning applications. Due to the relation \[
|X_i-X_j|^2 = |X_i|^2 + |X_j|^2 - 2X_i^TX_j,
\]
for the model where $X_i \sim \mathcal{U}(S^{p-1})$, where $|X_i|\equiv 1$, distance kernels can be regarded as inner-product kernels. However, for the model where $X_i$ $\sim \mathcal{N}(0,p^{-1}I_p)$, the fluctuations in $|X_i|$'s do seem to make a difference,  and so far we have not been able to draw any conclusion about the limiting spectrum.

Another extension is to kernels that are of more general forms, neither an inner-product kernel nor a distance one. For example, a complex-valued kernel has been used in \cite{singer2011classaveraging} for a dataset of tomographic images. Every pair of images is brought into in-plane rotational alignment. The modulus of the kernel function corresponds to the similarity of the images when they are optimally aligned, while the phase of the kernel is the optimal in-plane alignment angle. Notice that this kernel is discontinuous, since a small perturbation in the images may lead to a completely different phase. Similar kernels with discontinuity have also been used for dimensionality reduction \cite{singer2011vectordiffusionmap} and sensor network localization \cite{cucuringu2011sensor}. In many senses, these applications have been the motivation for the analysis presented in this paper.  

Finally, it is also possible to extend the study to non-Hermitian matrices as follows. Suppose that $X_1, \cdots, X_m$ are $m$ i.i.d random vectors in $\mathbb{R}^p$, and $Y_1, \cdots, Y_n$ are $n$ i.i.d random vectors in $\mathbb{R}^p$, independent from the $X_i$'s. The $m$-by-$n$ matrix $A$ is constructed as $A_{ij} = f(X_i^T Y_j)$ where $f$ is some function. The distribution of the singular values of $A$ in the limit $p,m,n\to \infty$ and $p/n=\gamma_1, p/m = \gamma_2$ is conjectured to converge to a certain limiting density.

\section*{Acknowledgements}
The authors would like to thank Charles Fefferman, Ramon van Handel, Yakov G. Sinai and Van Vu for several useful discussions. The authors were partially supported by by Award Number DMS-0914892 from the NSF and by Award Number R01GM090200 from the NIGMS. A. Singer was also supported by Award Number FA9550-09-1-0551 from AFOSR and by the Alfred P. Sloan Foundation. 

\appendix

\section{Solution of the Equation of $m(z)$}\label{sec:mzeqnsolution}
We rewrite Eq. (\ref{eq:mzeqn}) as
\begin{equation}\label{eq:mzeqncubic}
\frac{a(\nu -a^2)}{\gamma}m^3 + (\nu +az)m^2 + (a+\gamma z)m + \gamma = 0, \quad \Im(z)>0, \Im(m)>0,
\end{equation}
where $ a^2 \leq \nu$. When $a = 0$ ($a^2 = \nu$) the equation corresponds to the semi-circle distribution (M.P. distribution), and the existence and uniqueness of the solution with positive imaginary part are known. We consider the case where $ 0 < a^2 < \nu$, thus the cubic term in Eq. (\ref{eq:mzeqncubic}) does not vanish.

\begin{lemma}\label{lemma:unique}
For every $z$ with $\Im(z) > 0$, there exists a unique $m$ with $\Im(m) > 0$ for which Eq. (\ref{eq:mzeqncubic}) holds.
\end{lemma}
\begin{proof}
It can be verified that whenever $a,\nu, \gamma$ are real and $\Im(z)>0$, the solution $m$ must not be real. Define the domain ${\cal D} := \{(a,\nu,\gamma, z), \gamma >0, 0 < a^2 < \nu , \Im(z) > 0\}$ which has two connected components ${\cal D}_+ = {\cal D} \cap \{a>0\} $ and ${\cal D}_- = {\cal D} \cap \{a<0\} $. The three solutions of the cubic equation depend continuously on the coefficients, thus if we let $(a,\nu, \gamma, z)$ vary continuously in ${\cal D}_+$, the imaginary parts of the three solutions never change sign, and similarly for ${\cal D}_-$. As a result, it suffices to show that for one choice of $(a,\nu,\gamma, z)\in {\cal D}_+$ and one choice in ${\cal D}_-$, there is a unique solution with positive imaginary part. This can be done, for example, by choosing $ a = \pm 1/2$, $\nu = 1$, $\gamma = 1$ and $z = i$. 
\end{proof}

\begin{figure}
\centering
\includegraphics[width = 0.6\linewidth]{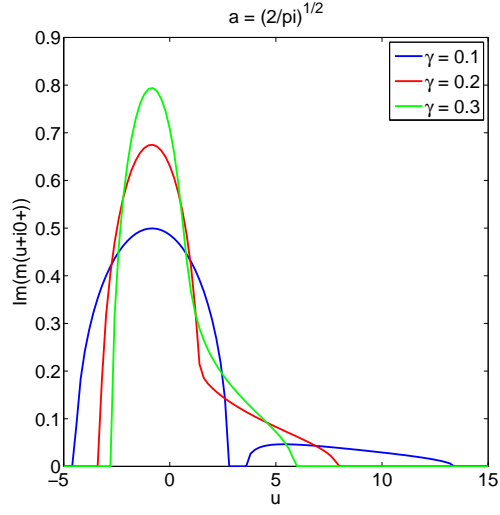}
\caption{\label{fig:mzsolution}
Function $y(u;a, \nu, \gamma)$ as in Eq. (\ref{eq:y(u)}).
}
\end{figure}

The explicit expression for $y(u)$ defined in Eq. (\ref{eq:y(u)def}) is given by
\begin{equation}\label{eq:y(u)}
y(u;a, \nu, \gamma)=\begin{cases}
0, & D\leq 0,\\
\frac{\sqrt{3}}{2}((\sqrt{D}+R)^{\frac{1}{3}} + (\sqrt{D}-R)^{\frac{1}{3}}), & D >0,
\end{cases}
\end{equation}
where\[
\begin{split}
D & = Q^3+R^2,\\
R & =  (9\alpha_2\alpha_1-27\alpha_0-2\alpha_2^3)/54,\\
Q & = (3\alpha_1 -\alpha_2^2)/9,
\end{split}
\]
and \[
m^3 + \alpha_2 m^2 + \alpha_1 m + \alpha_0 =0
\]
is derived from Eq. (\ref{eq:mzeqncubic}) by multiplying $(\frac{a\nu}{\gamma})^{-1}$on both sides. Explicitly,
\begin{eqnarray*}
\alpha_2 &= \frac{(\nu +au)\gamma}{a(\nu - a^2)}, \\
\alpha_1 &= \frac{(a+\gamma u)\gamma}{a(\nu - a^2)}, \\
\alpha_0 &= \frac{\gamma^2}{a(\nu - a^2)}.
\end{eqnarray*}
So all of $\alpha_2$, $\alpha_1$, $\alpha_0$, and thus $R$, $Q$ and $D$ are real numbers. $D$ is the ``discriminant" of cubic equation, where $D$ turning from negative to positive signals the emergence of a pair of complex solutions. The function $y(u; a, \nu, \gamma)$ is plotted in Fig. \ref{fig:mzsolution} where $\nu=1$, $a=\sqrt{2/\pi}$ and $\gamma = 0.1, 0.2, 0.3$. Notice the invariance of Eq. (\ref{eq:mzeqncubic}) under the transformation
\[
\nu c^2 \to \nu,~ ac \to a, ~ zc \to z, ~ m/c \to m
\] where $c$ is any positive constant, which corresponds to multiplying the kernel function by $c$.

\section{Lemma in Sec. 2}\label{sec:A2}

\begin{proof}[Proof of Lemma \ref{lemma:m_A-Em_A}] 
We need the Burkholder's Inequality (Lemma 2.12.
of \cite{bai}), which says that for $\{\gamma_{k},1\le k\le n\}$
being a (complex-valued) martingale difference sequence, for $\beta>1$,
\begin{equation}
\mathbb{E}|\sum_{k=1}^{n}\gamma_{k}|^{\beta}\leq K_{\beta}\mathbb{E}\left(\sum_{k=1}^{n}|\gamma_{k}|^{2}\right)^{\beta/2},
\label{eq:Burkholder0}
\end{equation}
where $K_\beta$ is a positive constant depending on $\beta$.
Using the i.i.d. random vectors $\{X_{i},1\leq i\leq n\}$, we will
define the martingale to be
\[
M_{k}=\mathbb{E}(\Tr(A-zI)^{-1}|\sigma\{X_{k+1},\cdots,X_{n}\}):=\mathbb{E}_{k}\Tr(A-zI)^{-1},\quad0\leq k\leq n,
\]
where $\sigma\{X_{k+1},\cdots,X_{n}\}:={\cal F}_{n-k}$ denotes the
$\sigma$-algebra generated by $\{X_{i},k+1\le i\le n\}$ and $\mathbb{E}(\cdot|{\cal G})$
the conditional expectation with respect to the sub-$\sigma$-algebra
${\cal G}$. We have $M_{n}=\mathbb{E}\Tr(A-zI)^{-1}$ and $M_{0}=\Tr(A-zI)^{-1}$,
and $M_{n},\cdots,M_{0}$ form an martingale with respect to the filtration
$\{{\cal F}_{t},t=0,\cdots,n\}$. The martingale difference 
\begin{align}
\gamma_{k} & =M_{k-1}-M_{k}\nonumber \\
 & =\mathbb{E}_{k-1}\Tr(A-zI)^{-1}-\mathbb{E}_{k}\Tr(A-zI)^{-1}\nonumber \\
 & =\mathbb{E}_{k}(\Tr(A-zI)^{-1}-\Tr(A^{(k)}-zI)^{-1}) \nonumber\\
 & \quad -\mathbb{E}_{k-1}(\Tr(A-zI)^{-1}-\Tr(A^{(k)}-zI)^{-1})
\label{eq:gammak_mart_diff}
\end{align}
where $A^{(k)}$ is an $(n-1)$-by-$(n-1)$ matrix that is obtained from the matrix $A$ by eliminating its $k$-th column and $k$-th row. Notice that $A^{(k)}$
is independent of $X_{k}$, 
$\mathbb{E}_{k-1}\Tr(A^{(k)}-zI)^{-1}=\mathbb{E}_{k}\Tr(A^{(k)}-zI)^{-1}$,
which verifies the last line of Eq. (\ref{eq:gammak_mart_diff}).
At the same time, we have 
\begin{equation}
|\Tr(A-zI)^{-1}-\Tr(A^{(k)}-zI)^{-1}|\le\frac{4}{v},
\label{eq:Tr(A-z)^-1-Tr(A^(k)-z)^-1}
\end{equation}
where $v=\Im(z)>0$, using an argument similar to that in Sec. 2.4. of \cite{Tao2011book} (see Eq. (2.96)). The way to show Eq. (\ref{eq:Tr(A-z)^-1-Tr(A^(k)-z)^-1}) is by making use of 
(1) that the ordered $n-1$ eigenvalues of a minor of a symmetric (or Hermitian) matrix $A$ `interlace' the ordered $n$ eigenvalues of $A$, which follows from the Courant-Fischer theorem (see, for example, Exercise 1.3.14 of \cite{Tao2011book}), and 
(2) that for fixed $z$ both real and imaginary parts of $(t-z)^{-1}$ as functions of $t$ have bounded total variation. As a result, 
\begin{equation*}
\begin{split}
|\gamma_{k}| 
& \le|\mathbb{E}_{k}(\Tr(A-zI)^{-1}-\Tr(A^{(k)}-zI)^{-1})| \\
& \quad +|\mathbb{E}_{k-1}(\Tr(A-zI)^{-1}-\Tr(A^{(k)}-zI)^{-1})| \\
& \le 2\frac{4}{v}:=C,
\end{split}
\end{equation*}
and then with Eq. (\ref{eq:Burkholder0}), choosing $\beta=4$,
\begin{align*}
\mathbb{E}|m_{A}-\mathbb{E}m_{A}|^{4} & =\frac{1}{n^{4}}\mathbb{E}|\sum_{k=1}^{n}\gamma_{k}|^{4}\\
 & \le\frac{1}{n^{4}}K_{4}\left(\sum_{k=1}^{n}|\gamma_{k}|^{2}\right)^{2}\\
 & \le\frac{1}{n^{4}}K_{4}(nC^{2})^{2}={\cal O}(1)n^{-2}.
\end{align*} 
This implies the almost sure convergence of $m_{A}-\mathbb{E}m_{A}$ to 0 by Borel-Cantelli lemma. Also, Eq. (\ref{eq:Lemma.|m_A-Em_A|<n^-1/2}) follows by Jensen's inequality.
\end{proof}

\begin{lemma}\label{lemma:E|m_A-RHS(tildem)|}
 Notations as in Sec. \ref{subsec:A1isMP}\[
\mathbb{E}\left|m_{A}(z)-\left(-z-\left(1-\left(1+\frac{1}{\gamma}\mathbb{E}\tilde{m}(z)\right)^{-1}\right)\right)^{-1}\right|\rightarrow 0.
\]
\end{lemma}
\begin{remark} \label{rk:A1_1}
The proof provided below can be replaced by a simpler one. The reason we give this proof is that it contains many of the techniques that are used in showing the main result. 
\end{remark}

\begin{proof}
Continue from Eq. (\ref{eq: A1denominator}).  
We first observe that when $p$ is large, $|X_{n}|^{2}$ concentrates at 1, 
and specifically, with $p$ large enough
\begin{equation}
\Pr \left[ | |X_{n}|^{2}-1 |>\sqrt{\frac{40\ln p}{p}} \right]<p^{-9},
\label{bound.|Xn|^1-1isdelta_p-9}
\end{equation}
which can be verified by standard large deviation inequality techniques. However, at this stage the following moment bound will be enough for our purpose:
\begin{equation}
\mathbb{E}\left||X_{n}|^{2}-1\right|
\le\sqrt{\mathbb{E}\left(|X_{n}|^{2}-1\right)^{2}}
=\sqrt{\frac{2}{p}}\rightarrow0.
\label{eq:Xn2is1}
\end{equation}
We then write the denominator in Eq.  (\ref{eq: A1denominator}) as
\begin{equation}
\begin{split}
\eta^{T}(\tilde{A}^{(n)}-D_{\eta}-zI_{n-1})^{-1}\eta
&= \frac{1}{p}\mathbf{Tr}(\tilde{A}^{(n)}-zI_{n-1})^{-1}+r \\
&= \frac{1}{\gamma} \mathbb{E} \tilde{m}(z) + \frac{1}{\gamma}\tilde{r} + r ,
\end{split}
\label{eq:eta(tildeA-D-z)^-1eta}
\end{equation}
 where $r = \eta^{T}(\tilde{A}^{(n)}-D_{\eta}-zI_{n-1})^{-1}\eta - \frac{1}{p}\mathbf{Tr}(\tilde{A}^{(n)}-zI_{n-1})^{-1}  $,  $\tilde{m}(z): =\frac{1}{n}\Tr(\tilde{A}^{(n)}-zI_{n-1})^{-1}$,
and $\tilde{r}:=(\tilde{m}(z)-\mathbb{E}\tilde{m}(z))$.
We have that
\begin{enumerate}
\item $\mathbb{E}|\tilde{r}|\le{\cal O}(1)n^{-1/2}$
as $n\rightarrow\infty$: Because $\tilde{A}^{(n)}$ is itself an $(n-1)\times(n-1)$ kernel matrix
by Eq. (\ref{eq:tildeA1}), Lemma \ref{lemma:m_A-Em_A} applies.

\item 
 $r$ splits into two terms
\begin{align*}
r & =\left(\eta^{T}(\tilde{A}^{(n)}-D_{\eta}-zI_{n-1})^{-1}\eta-\eta^{T}(\tilde{A}^{(n)}-zI_{n-1})^{-1}\eta\right)\\
 & \quad +\left(\eta^{T}(\tilde{A}^{(n)}-zI_{n-1})^{-1}\eta-\frac{1}{p}\mathbf{Tr}(\tilde{A}^{(n)}-zI_{n-1})^{-1}\right)\\
 & :=r_{1}+r_{2},
\end{align*}
where  (1) $ \mathbb{E}|r_{2}|\le{\cal O}(1)p^{-1/2}$, by Lemma \ref{bound:E|r_2|forA1}; 
(2) $|r_{1}|{\bf 1}_{\Omega_{\delta}}\le{\cal O}(1)p^{-1/2}$, where $\Omega_{\delta}$ is a large probability set  depending on $p$, defined as 
\[
\Omega_{\delta}=\{|\eta_{i}|<\delta,\,1\le i\le n-1,\,\delta=\frac{M}{\sqrt{p}}\}, 
\quad M=\sqrt{20\ln p},
\]
by Lemma \ref{lemma:r1Omegadelta}. Notice that $M=o(p^{\epsilon})$ for any $\epsilon>0$.
\end{enumerate}

Back to Eq. (\ref{eq:EmA1}). By Eqs. (\ref{eq: A1denominator}) and (\ref{eq:eta(tildeA-D-z)^-1eta}), 
we have 
\begin{align*}
\mathbb{E}m_{A}(z) & =\mathbb{E}\left((A-zI)^{-1}\right)_{nn}\\
 & =\mathbb{E}\left(-z-|X_{n}|^{2}\left(1-(1+\frac{1}{p}\mathbf{Tr}(\tilde{A}^{(n)}-zI_{n-1})^{-1}+r)^{-1}\right)\right)^{-1}.
\end{align*}
The following bounds (1) - (4) can be verified:
\begin{itemize}

\item[(1)] (Lemma \ref{bound:|eta(A-z)^-1eta|<M'})
 On $\Omega_{\delta}$, $|\eta^{T}(A^{(n)}-zI_{n-1})^{-1}\eta|$
and $|(1+\eta^{T}(\tilde{A}^{(n)}-D_{\eta}-zI_{n-1})^{-1}\eta)^{-1}|$
are both bounded by $M^{'}=1+{\cal O}(1)M^{2}$, $M^{'}=o(p^{\epsilon})$
for any $\epsilon>0$. 

\item[(2)] 
(Lemma \ref{bound: |(1+gamma^-1(Em))^-1|})
On $\Omega_{r}\cap\Omega_{\delta}$, 
$\left|\left(1+\frac{1}{\gamma}\mathbb{E}\tilde{m}(z)\right)^{-1}\right|\le2M^{'}$, 
where we define 
\[
\Omega_{r}=\{|\tilde{r}|<p^{-1/4},|r_{2}|<p^{-1/4}\},
\]
and by Markov inequality, we have 
\[
\Pr(\Omega_{r}^{c})\le p^{1/4}\mathbb{E}|\tilde{r}|+p^{1/4}\mathbb{E}|r_{2}|\le{\cal O}(1)p^{-1/4}
\]
 when $p$ is large.

\item[(3)] $\left|\left((A-zI)^{-1}\right)_{nn}\right|\le\frac{1}{v}$, which is Eq. (\ref{eq:boundii}).

\item[(4)] $\left|\left(-z-\left(1-\left(1+\frac{1}{\gamma}\mathbb{E}\tilde{m}(z)\right)^{-1}\right)\right)^{-1}\right|\le\frac{1}{v}$:
By $\Im\left(-\left(1+\frac{1}{\gamma}\mathbb{E}\tilde{m}(z)\right)^{-1}\right)$
equals a positive number times $\Im\left(\frac{1}{\gamma}\mathbb{E}\tilde{m}(z)\right)$
which is also positive, one verifies that \[
\Im\left(-z-\left(1-\left(1+\frac{1}{\gamma}\mathbb{E}\tilde{m}(z)\right)^{-1}\right)\right)<\Im(-z)=-v,
\] so \[
\left|-z-\left(1-\left(1+\frac{1}{\gamma}\mathbb{E}\tilde{m}(z)\right)^{-1}\right)\right|>v.
\] 
\end{itemize}

With (1) and (2), we have 
\begin{align}
  \mathbb{E}
 & \left|\left(1 +\eta^{T}(\tilde{A}^{(n)}-D_{\eta}-zI_{n-1})^{-1}\eta\right)^{-1}-\left(1+\frac{1}{\gamma}\mathbb{E}\tilde{m}(z)\right)^{-1}\right|\cdot{\bf 1}_{\Omega_{\delta}\cap\Omega_{r}}\nonumber \\
 & \le\mathbb{E}(M^{'}\cdot2M^{'})(|r|+\frac{1}{\gamma}|\tilde{r}|)\cdot{\bf 1}_{\Omega_{\delta}\cap\Omega_{r}}\nonumber \\
 & \le2M^{'2}(\mathbb{E}|r_{2}|+\mathbb{E}|r_{1}|\cdot{\bf 1}_{\Omega_{\delta}}+\gamma^{-1}\mathbb{E}|\tilde{r}|)\nonumber \\
 & \le2M^{'2}({\cal O}(1)p^{-1/2}+{\cal O}(1)p^{-1/2}+{\cal O}(1)n^{-1/2}) \nonumber \\
 & ={\cal O}(1)M^{'2}p^{-1/2}.\label{eq:|()^-1-()^-1|inA1}
\end{align}
Using bounds (1)-(4), together with Eqs. (\ref{eq:|()^-1-()^-1|inA1}) and (\ref{eq:Xn2is1}), we have
\begin{align*}
 & \mathbb{E}\left|m_{A}(z)-\left(-z-\left(1-\left(1+\frac{1}{\gamma}\mathbb{E}\tilde{m}(z)\right)^{-1}\right)\right)^{-1}\right|\\
 & =\mathbb{E}\left|\left(-z-|X_{n}|^{2}\eta^{T}(A^{(n)}-zI_{n-1})^{-1}\eta\right)^{-1}-\left(-z-\left(1-\left(1+\frac{1}{\gamma}\mathbb{E}\tilde{m}(z)\right)^{-1}\right)\right)^{-1}\right|\\
 & \le\frac{2}{v}\left(\Pr(\Omega_{\delta}^{c})+\Pr(\Omega_{r}^{c})\right)\\
 & \quad + \mathbb{E}\left|\left(-z-|X_{n}|^{2}\eta^{T}(A^{(n)}-zI_{n-1})^{-1}\eta\right)^{-1}-\left(-z-\left(1-\left(1+\frac{1}{\gamma}\mathbb{E}\tilde{m}(z)\right)^{-1}\right)\right)^{-1}\right|\cdot{\bf 1}_{\Omega_{\delta}\cap\Omega_{r}}\\
 & \le\frac{2}{v}\left(\Pr(\Omega_{\delta}^{c})+\Pr(\Omega_{r}^{c})\right)\\
 & \quad + \mathbb{E}\frac{1}{v^{2}}\left||X_{n}|^{2}-1\right|\cdot|\eta^{T}(A^{(n)}-zI_{n-1})^{-1}\eta|\cdot{\bf 1}_{\Omega_{\delta}\cap\Omega_{r}}\\
 & \quad + \mathbb{E}\frac{1}{v^{2}}\left|\left(1+\eta^{T}(\tilde{A}^{(n)}-D_{\eta}-zI_{n-1})^{-1}\eta\right)^{-1}-\left(1+\frac{1}{\gamma}\mathbb{E}\tilde{m}(z)\right)^{-1}\right|\cdot{\bf 1}_{\Omega_{\delta}\cap\Omega_{r}}\\
 & \le\frac{2}{v}\left(\Pr(\Omega_{\delta}^{c})+\Pr(\Omega_{r}^{c})\right)+\mathbb{E}\frac{1}{v^{2}}\left||X_{n}|^{2}-1\right|M^{'}{\bf 1}_{\Omega_{\delta}\cap\Omega_{r}}+\frac{1}{v^{2}}{\cal O}(1)M^{'2}p^{-1/2}\\
 & \le{\cal {\cal O}}(1)p^{-9}+{\cal O}(1)p^{-1/4}+M^{'}{\cal O}(1)p^{-1/2}+{\cal O}(1)M^{'2}p^{-1/2}\\
 & =o(p^{-1/2+\epsilon}),
\end{align*}
 for any $\epsilon>0$, which proves the statement.
\end{proof}

\begin{lemma}\label{lemma:r1Omegadelta}
Notations as in Lemma \ref{lemma:E|m_A-RHS(tildem)|},
\[
|r_{1}|{\bf 1}_{\Omega_{\delta}}\le{\cal O}(1)p^{-1/2}
\] 
\end{lemma}
\begin{proof}
By 
\begin{align}
\Pr[|\eta_{i}|>\delta] 
&=  2\int_{M}^{\infty}\frac{1}{\sqrt{2\pi}}e^{-\frac{u^{2}}{2}}du \nonumber \\
& \le \frac{1}{\sqrt{2}}e^{-\frac{M^{2}}{2}} \nonumber \\
&= \frac{1}{\sqrt{2}}p^{-10},\quad1\le i\le n-1, \label{eq:eta_delta}
\end{align}
 and the union bound, we have 
\[
\Pr(\Omega_{\delta}^{c})\le(n-1)\Pr[|\eta_{i}|>\delta]\le{\cal O}(1)p^{-9}.
\]
 Now (recall that $s(\cdot)$ denotes the magnitude of the largest singular value/spectral norm of a matrix)
\begin{align*}
|r_{1}| 
& =\left|\eta^{T}(\tilde{A}^{(n)}-D_{\eta}-zI_{n-1})^{-1}\eta-\eta^{T}(\tilde{A}^{(n)}-zI_{n-1})^{-1}\eta\right|\\
 & =\left|\eta^{T}(\tilde{A}^{(n)}-D_{\eta}-zI_{n-1})^{-1}D_{\eta}(\tilde{A}^{(n)}-zI_{n-1})^{-1}\eta\right|\\
 & \le s\left((\tilde{A}^{(n)}-D_{\eta}-zI_{n-1})^{-1}D_{\eta}(\tilde{A}^{(n)}-zI_{n-1})^{-1}\right)|\eta|^{2}.
\end{align*}
Notice that on $\Omega_{\delta}$ 
\[
s(D_{\eta})\le\max_{1\le i\le n-1}\eta_{i}^{2}\le\delta^{2},
\]
also $|\eta|^{2}\le(n-1)\delta^{2}$. At the same time both $s((\tilde{A}^{(n)}-D_{\eta}-zI_{n-1})^{-1})$ and $s((\tilde{A}^{(n)}-zI_{n-1})^{-1})$ is bounded by $\frac{1}{v}$ an absolute constant. Adding together (for Hermitian matrices $A$ and $B$, $s(AB)\le s(A)s(B)$) we have 
\begin{equation}
|r_{1}|{\bf 1}_{\Omega_{\delta}}\le\frac{1}{v^{2}}\delta^{2}\cdot(n-1)\delta^{2}=\frac{M^{4}(n-1)}{v^{2}p^{2}}<{\cal O}(1)p^{-1/2}.
\label{eq:r1bound}
\end{equation}
\end{proof}

\begin{lemma} \label{bound:E|r_2|forA1}
Notations as in Sec. \ref{subsec:A1isMP},
\[
 \mathbb{E}|r_{2}|\le{\cal O}(1)p^{-1/2}.
\]
\end{lemma}
\begin{remark}
The technique is similar to the moment bound method in \cite[Chapter 3.3]{bai}, where the main observation is that $\tilde{A}^{(n)}$ is independent of the vector $\eta$.
\end{remark}
\begin{proof} 
Define $(\tilde{A}^{(n)}-zI_{n-1})^{-1}$ as $\tilde{B}$ which is Hermitian, we have 
\begin{align*}
\mathbb{E}|r_{2}|^{2} & =\mathbb{E}\left|\sum_{i=1}^{n-1}\left(\eta_{i}^{2}-\frac{1}{p}\right)\tilde{B}_{ii}+\sum_{i_{1}\ne i_{2}}\eta_{i_{1}}\eta_{i_{2}}\tilde{B}_{i_{1}i_{2}}\right|^{2} \\
 & =\mathbb{E}\sum_{i,i^{'}}\left(\eta_{i}^{2}-\frac{1}{p}\right)\left(\eta_{i^{'}}^{2}-\frac{1}{p}\right)\tilde{B}_{ii}\overline{\tilde{B}_{i^{'}i^{'}}}\label{eq:r2_1}\\
 & \quad +\mathbb{E}\sum_{i}\sum_{i_{1}\ne i_{2}}\left(\eta_{i}^{2}-\frac{1}{p}\right)\eta_{i_{1}}\eta_{i_{2}}\left(\tilde{B}_{ii}\overline{\tilde{B}_{i_{1}i_{2}}}+\overline{\tilde{B}_{ii}}\tilde{B}_{i_{1}i_{2}}\right) \\
 & \quad + \mathbb{E}\sum_{i_{1}\ne i_{2}}\sum_{i_{1}^{'}\ne i_{2}^{'}}\eta_{i_{1}}\eta_{i_{2}}\eta_{i_{1}^{'}}\eta_{i_{2}^{'}}\tilde{B}_{i_{1}i_{2}}\overline{\tilde{B}_{i_{1}^{'}i_{2}^{'}}}. 
\end{align*}
By taking expectation over $\eta_{i}$'s first, we see many terms vanish due to the independence of $\eta_{i_1}$ and $\eta_{i_2}$ for $i_{1}\neq i_{2}$, and what remains gives 
\begin{align*}
\mathbb{E}|r_{2}|^{2} 
& \le\mathbb{E} \left( \sum_{i}\mathbb{E}\left(\eta_{i}^{2}-\frac{1}{p}\right)^{2}|\tilde{B}_{ii}|^{2}+\sum_{i_{1}\ne i_{2}}2\frac{1}{p^{2}}|\tilde{B}_{i_{1}i_{2}}|^{2} \right) \\
 & =\mathbb{E}\frac{2}{p^{2}}\left(\sum_{i}|\tilde{B}_{ii}|^{2}+\sum_{i_{1}\ne i_{2}}|\tilde{B}_{i_{1}i_{2}}|^{2}\right) \\
 & =\mathbb{E}\frac{2}{p^{2}}\mathbf{Tr}(\overline{\tilde{B}}^{T}\tilde{B}).
\end{align*}
Observe that 
\[
\mathbf{Tr}(\overline{\tilde{B}}^{T}\tilde{B})=\sum_{i=1}^{n-1}\frac{1}{|\tilde{\lambda}_{i}-z|^{2}}\le\sum_{i=1}^{n-1}\frac{1}{v^{2}}=\frac{n-1}{v^{2}},
\]
 where $v=\Im(z)>0$ and $\tilde{\lambda_{i}}$ are the eigenvalues of $\tilde{A}^{(n)}$. Then 
\[
\frac{2}{p^{2}}\mathbf{Tr}(\overline{\tilde{B}}^{T}\tilde{B})\le\frac{2}{v^{2}}\frac{n-1}{p^{2}}\le\frac{2}{v^{2}\gamma}\cdot\frac{1}{p},
\]
which means that  
\[
\mathbb{E}|r_{2}|^{2}\le\frac{\mathcal{O}(1)}{p},
\]
so we have $ \mathbb{E}|r_{2}|\le\sqrt{\mathbb{E}|r_{2}|^{2}}\le \mathcal{O}(1)p^{-1/2}$ . 
\qedhere
\end{proof}

\begin{lemma}\label{bound:E|m_A-m_tildeA|}
Notations as in Sec. \ref{subsec:A1isMP},
\[\mathbb{E}|m_{A}(z)-\tilde{m}(z)|\rightarrow0.
\]
\end{lemma}
\begin{proof}
First, $|m_{A}(z)-m_{A^{(n)}}(z)|\le\frac{4}{v}\cdot n^{-1}\to 0$, due to Eq. (\ref{eq:Tr(A-z)^-1-Tr(A^(k)-z)^-1}). 
Second, we show that $\mathbb{E}|m_{A^{(n)}}-m_{\tilde{A}^{(n)}}|\rightarrow0$. By
\begin{align*}
 & \Tr(A^{(n)}-zI_{n-1})^{-1}-\Tr(\tilde{A}^{(n)}-zI_{n-1})^{-1}\\
 & =\Tr(-(A^{(n)}-zI_{n-1})^{-1}(\eta\eta^{T}-D_{\eta})(\tilde{A}^{(n)}-zI_{n-1})^{-1})\\
 & =-\eta^{T}(A^{(n)}-zI_{n-1})^{-1}(\tilde{A}^{(n)}-zI_{n-1})^{-1}\eta \\
 & \quad +\Tr((A^{(n)}-zI_{n-1})^{-1}D_{\eta}(\tilde{A}^{(n)}-zI_{n-1})^{-1}),
\end{align*}
and using a similar argument as before, we can show that on $\Omega_{\delta}$
\[
\left|\eta^{T}(A^{(n)}-zI_{n-1})^{-1}(\tilde{A}^{(n)}-zI_{n-1})^{-1}\eta\right|\le\frac{1}{v^{2}}|\eta|^{2}\le\frac{1}{v^{2}}(n-1)\delta^{2}\le{\cal O}(1)M^{2},
\]
 and 
\begin{align*}
 & \left|\Tr((A^{(n)}-zI_{n-1})^{-1}D_{\eta}(\tilde{A}^{(n)}-zI_{n-1})^{-1})\right|\\
 & \le(n-1)s((A^{(n)}-zI_{n-1})^{-1}D_{\eta}(\tilde{A}^{(n)}-zI_{n-1})^{-1})\\
 & \le\frac{1}{v^{2}}(n-1)\delta^{2}={\cal O}(1)M^{2}.
\end{align*}
As a result, 
\begin{align*}
\mathbb{E}|m_{A^{(n)}}-m_{\tilde{A}^{(n)}}| & =\frac{2}{v}\Pr(\Omega_{\delta}^{c})+\mathbb{E}|m_{A^{(n)}}-m_{\tilde{A}^{(n)}}|\cdot{\bf 1}_{\Omega_{\delta}}\\
 & \le{\cal O}(1)p^{-9}+\frac{1}{n}{\cal O}(1)M^{2}
\end{align*}
 which goes to 0 as $n,p\rightarrow\infty$ with $p/n=\gamma$. 
\end{proof}

\begin{lemma}\label{bound: |(1+gamma^-1(Em))^-1|}
Notations as in Sec. \ref{subsec:A1isMP}, on $\Omega_{r}\cap\Omega_{\delta}$,
\[
\left|\left(1+\frac{1}{\gamma}\mathbb{E}\tilde{m}(z)\right)^{-1}\right| \le2M^{'}.
\]
\end{lemma}
\begin{proof}
On $\Omega_{r}\cap\Omega_{\delta}$,
with Eq. (\ref{eq:r1bound}) $|r_{1}|<{\cal O}(1)p^{-1/2}$ thus
$|r|\le|r_{1}|+|r_{2}|$ is bounded by ${\cal O}(1)p^{-1/4}$, 
\begin{align*}
\left|\frac{1}{1+\frac{1}{\gamma}\mathbb{E}\tilde{m}(z)}\right| & =\left|\frac{1}{1+\eta^{T}(\tilde{A}^{(n)}-D_{\eta}-zI_{n-1})^{-1}\eta-r-\frac{1}{\gamma}\tilde{r}}\right|\\
 & \le\frac{1}{\left|1+\eta^{T}(\tilde{A}^{(n)}-D_{\eta}-zI_{n-1})^{-1}\eta\right|-|r|-\frac{1}{\gamma}|\tilde{r}|}\\
 & \le2M^{'}
\end{align*}
 as $\left|1+\eta^{T}(\tilde{A}^{(n)}-D_{\eta}-zI_{n-1})^{-1}\eta\right|\ge1/M^{'}\gg(|r|+\frac{1}{\gamma}|\tilde{r}|)$,
where the latter is bounded by ${\cal O}(1)p^{-1/4}$. 
\end{proof}

\begin{lemma}\label{bound:|eta(A-z)^-1eta|<M'} 
Notation as in Sec. \ref{subsec:A1isMP}, on $\Omega_{\delta}$, 
both $|\eta^{T}(A^{(n)}-zI_{n-1})^{-1}\eta|$ and 
$|(1+\eta^{T}(\tilde{A}^{(n)}-D_{\eta}-zI_{n-1})^{-1}\eta)^{-1}|$ 
are bounded by $M^{'}$.
\end{lemma}
\begin{proof}
On $\Omega_{\delta}$, 
$|\eta^{T}(A^{(n)}-zI_{n-1})^{-1}\eta|\leq s((A^{(n)}-zI_{n-1})^{-1})|\eta|^{2}\leq\frac{1}{v}\delta^{2}(n-1)
={\cal O}(1)M^{2}$,
and also
\begin{align*}
 & \left|\left(1+\eta^{T}(\tilde{A}^{(n)}-D_{\eta}-zI_{n-1})^{-1}\eta\right)^{-1}\right|\\
 & =\left|1-\eta^{T}(\eta\eta^{T}+\tilde{A}^{(n)}-D_{\eta}-zI_{n-1})^{-1}\eta\right|\\
 & \le1+|\eta^{T}(A^{(n)}-zI_{n-1})^{-1}\eta|\\
 & \le1+{\cal O}(1)M^{2}:=M^{'}.\qedhere
\end{align*}
\end{proof}

\section{Lemma in Sec. 3}\label{sec:A3}
\begin{lemma}\label{lemma:alptoalN}
Model and notations as in Sec. \ref{subsec:modelnotation}. Due to Eqn. (\ref{eq:Exi^k}), the result in Lemma \ref{lemma:Plpmatchhl} holds.

Suppose that $k(x;p)$ is in ${\cal H}_{\mathcal{N}}$ and ${\cal H}_p$ for all $p$, and satisfies\[
\int_{\mathbb{R}} k(x;p)^2 |q_p(x)- q(x)| dx \to 0, \quad p \to \infty.
\]
Let 
\[\begin{split}
b_{l,p} & = \int_{\mathbb{R}} k(x;p)h_l(x) q(x) dx, \\
a_{l,p} & = \int_{\mathbb{R}} k(x;p) P_{l,p}(x) q_p(x) dx,
\end{split}
\] for $ l = 0, 1, \cdots$. Then for each $l$, $|b_{l,p}- a_{l,p}| \to 0 $ as $p \to \infty$.
\end{lemma}

\begin{proof}
\[\begin{split}
& |b_{l,p}  -  a_{l,p}|  \\
&= |\int_\mathbb{R} kh_l (q-q_p) dx + \int_\mathbb{R} k(h_l-P_{l,p})q_p dx | \\
& \leq \int |k h_l| |q-q_p| dx + \int |k||h_l-P_{l,p}|q_p dx \\
& : = (1) + (2).
\end{split}
\]
For (1), by Cauchy-Swarchz\[
(1)^2 \leq (\int k^2 |q-q_p| dx)(\int h_l^2 |q-q_p| dx),
\] where
\[
\int h_l^2 |q-q_p| dx \leq \int h_l^2 q dx + \int h_l^2 q_p dx = 1 + (1 +\mathcal{O}_l(1)p^{-1}),
\] which is bounded as $p \to \infty$, and $\int k^2 |q-q_p| dx \to 0$, thus $(1) \to 0$. For (2), 
\[
(2)^2 \leq (\int k^2 q_p dx)(\int (h_l-P_{l,p})^2 q_p dx),
\]  where $\int k^2 q_p dx \to \int k^2 q dx $ which is bounded, and by Lemma \ref{lemma:Plpmatchhl} \[
h_l(x) - P_{l,p}(x) = \sum_{j=0}^l (\delta_{l,p})_j x^j , \quad \max_{0\leq j \leq l} |(\delta_{l,p})_j| < \mathcal{O}_l(1)p^{-1},
\] thus \[
\left(\int (h_l-P_{l,p})^2 q_p dx \right)^{1/2} \leq \mathcal{O}_l(1)p^{-1},
\] so $(2) \to 0$.
\end{proof}

\begin{lemma}\label{lemma:k(x;p)=k(x)}
Model and notations as in Sec. \ref{subsec:mainthm}, and suppose that $k(x)$ is as in Remark \ref{rk:k(x;p)=k(x)}.
Eqn. (\ref{eq:with|qp(x)-q(x)|}) implies that $ \mathbb{E}k(\xi_p)^2 \to \mathbb{E}k(\zeta)^2 = \nu_\mathcal{N}$. Without loss of generality, $k(x)$ is in ${\cal H}_p$ for all $p$. Define $b_{l,p}$ and $a_{l,p}$ as in Lemma \ref{lemma:alptoalN}, and notice that since $k(x)$ does not depend on $p$, $b_{l,p} = b_l$ independent of $p$.
  
Then conditions {\bf (C.Variance)}, {\bf (C.$p$-Unform)} and {\bf (C.$\alpha_1$)} are satisfied by $k(x;p) =  k(x)-a_{0,p} $. Also, $\nu_p  \to  \nu_\mathcal{N}$, and $a_{1,p}\to  a_\mathcal{N} = b_1$.
\end{lemma}
\begin{proof}
By definition $\mathbb{E}k(\xi_p;p) =0$. In this case,\[
\nu_p  = \mathbb{E}k(\xi_p;p)^2 = \mathbb{E}k(\xi_p)^2 - a_{0,p}^2.
\] Since Lemma \ref{lemma:alptoalN} applies to $k(x)$, we know that \[
a_{0,p}\to b_0 = \mathbb{E}k(\zeta) =0.
\] Together with the fact that $\mathbb{E}k(\xi_p)^2 \to \mathbb{E}k(\zeta)^2 = \nu_\mathcal{N}$, we know that $\nu_p  \to \nu_\mathcal{N}$ as $p \to \infty$. Thus {\bf (C.Variance)} is satisfied.

Also $a_{1,p} \to b_1$ which is a constant, thus {\bf (C.$\alpha_1$)} holds.

For  {\bf (C.$p$-Unform)}  to be satisfied, it suffices to show that $\sum_{l=L+1}^\infty a_{l,p}^2$ can be made $p$-uniformly small. Notice that 
\[\begin{split}
\sum_{l=0}^\infty a_{l,p}^2 &= \mathbb{E}k(\xi_p)^2 \to \nu_\mathcal{N},\\
\sum_{l=0}^\infty b_l^2     &= \nu_\mathcal{N},
\end{split}
\] and meanwhile for each $l$, $a_{l,p} \to b_l$ by Lemma \ref{lemma:alptoalN}, thus for any finite $L$ \[
\begin{split}
\sum_{l=L+1}^\infty a_{l,p}^2 
& = \mathbb{E}k(\xi_p)^2 - \sum_{l=0}^L a_{l,p}^2 \\
& \to  \nu_\mathcal{N} - \sum_{l=0}^L b_l^2 = \sum_{l=L+1}^\infty b_{l}^2,
\end{split}
\] which can be made small by choosing $L$ large independently of $p$.
\end{proof}

\begin{lemma}\label{lemma:fsmooth}
Notations as in Sec. \ref{subsec:mainthm}. If $f(\xi;p)=f(\xi)$ is $C^{1}$ at $\xi=0$, then the theorem applies and $a^2 = \nu$. Specifically, $a = f^{'}(0)$.
\end{lemma}
\begin{proof}
We first truncate $f(\xi)$ to be $\hat{f}(\xi;p) = f(\xi)\mathbf{1}_{\{|\xi| \leq \delta\}} $, where $\delta = \delta(p)=\frac{M}{\sqrt{p}}$, $M=\sqrt{20\ln p}$. Using a similar argument as in Lemma \ref{pf:omega_delta_p-9}, we have\[
\Pr [ \exists i\neq j, ~ |X_i^TX_j| > \delta ] \le {\cal O}(1)p^{-7}.
\] Thus, if we denote $\hat{A}$ as the random kernel matrix with kernel function $\hat{f}$, then for fixed $z=u+iv$ \[
\mathbb{E}|m_A(z) - m_{\hat{A}}(z)| \leq \frac{2}{v}\Pr [ \exists i\neq j, ~ |X_i^TX_j| > \delta ] \to 0,
\] where $m_A(z)$ and $m_{\hat{A}}(z)$ are the Stieltjes transforms of $A$ and $\hat{A}$ respectively. Since the convergence of $\mathbb{E}m_A(z)$ implies the convergence of the spectral density, if suffices to show that the claim in the lemma holds for $\hat{f}(\xi;p)$.

Since $f(\xi)$ is $C^1$ at $\xi =0$, for any $\epsilon > 0$ there exists a neighborhood $[-R, R]$  on which 
\[f(\xi) =f(0)+f^{'}(0)\xi+ r(\xi),\quad |r(\xi)| \leq \epsilon |\xi|.
\]  Since $\delta \to 0$ when $p \to \infty$, we assume that $p$ is large enough so that $\delta < R$. Let $k(x;p) = \sqrt{p}\hat{f}(x/\sqrt{p})$, and assume that $f(0) = 0$ since it only contributes to $\mathbb{E}k(\xi_p,p)=a_{0,p}$, we have \[
\begin{split}
k(x;p) 
&= \left(f^{'}(0)x + \sqrt{p}r(\frac{x}{\sqrt{p}}) \right)\mathbf{1}_{\{|x| \leq M\}}\\
&:= k_1 + k_2,
\end{split}
\] where $|k_2(x;p)| \leq \epsilon |x|$, so $\mathbb{E}k_2(\xi_p;p)^2 \leq \epsilon^2$. Thus, the $L^2$ norm of $k_2$ is arbitrarily small in ${\cal H}_p$, and $\nu_p = Var(k(\xi_p;p))$ and $a_{1,p} = \mathbb{E}\xi_p(k(\xi_p;p) - \mathbb{E}k(\xi_p;p))$ are decided by $k_1$. For $k_1(x;p) = f^{'}(0)x\mathbf{1}_{\{|x| \leq M\}}$, $\mathbb{E}k_1(\xi_p;p) =0$, and since $M\to \infty$ as $p \to \infty$, $\mathbb{E}k_1(\xi_p;p)^2 \to (f^{'}(0))^2 $ and $\mathbb{E}\xi_p k_1(\xi_p;p) \to f^{'}(0)$. Thus $\nu_p \to (f^{'}(0))^2 = \nu$, and $a_{1,p} \to f^{'}(0) = a$. 
 \end{proof}
 
\begin{lemma}\label{lemma:xipLargeDeviation}
Let $\xi_p$ be as in Sec. \ref{subsec:mainthm}, and equivalently $\xi_p = p^{-1/2}\sum_{i=1}^p x_iy_i$ where $x_i$ and $y_i$ i.i.d.$ \sim \mathcal{N}(0,1)$.
Then for $p > 2$, \[
\Pr[ |\xi_p| > R] \leq (2e)e^{-R}.
\]
\end{lemma}
\begin{proof}
Since for $|t| < \sqrt{p}$,  
$\mathbb{E}e^{t\frac{x_1y_1}{\sqrt{p}}} =  (1-t^{2}/p)^{-1/2}$, by choosing $t = 1$ we have
\[\begin{split}
\Pr[\xi_p > R] 
& \leq e^{-M}(\mathbb{E}e^{\frac{x_1y_1}{\sqrt{p}}})^p \\
&  =  e^{-M}(1-\frac{1}{p})^{-p/2} \\
& \leq e^{-M} e,
\end{split}
\] where the last line is due to that $x=1/p$ satisfies $\log(1-x)/x > -2$ when $0 < x < 1/2$. The argument  for bounding $\Pr[\xi_p < -R]$ is similar.
\end{proof}
 
 \begin{lemma}\label{lemma:with|qp-q|}
Notations as in Sec. \ref{subsec:mainthm}. Suppose $k(x)$ is 
(Case 1) bounded, or (Case 2) in ${\cal H}_\mathcal{N}$ and ${\cal H}_p$ for all $p$, is bounded on $|x|\leq R$ for any $R>0$, and satisfies \[
\int_{|x| > R} k(x)^2 q_p(x) dx \to 0, \quad R\to \infty
\]  uniformly in $p$, then Eqn. (\ref{eq:with|qp(x)-q(x)|}) holds.
  \end{lemma}
\begin{proof}
First, we reduce (Case 2) to (Case 1). Notice that 
\[\begin{split}
& \int_{\mathbb{R}} k(x)^2 |q_p(x)- q(x)| dx \\
  \leq & \int_{|x| \leq R} k(x)^2 |q_p(x)- q(x)| dx + \int_{|x| >R} k(x)^2 q_p(x)dx \\
& \quad +  \int_{|x| >R} k(x)^2 q(x) dx.
\end{split} \] 
The last two terms can be made arbitrarily small independently of $p$ by choosing $R$ large, and for fixed $R$, the first term goes to 0 given that (Case 1) is proved.
  
To show the claim for (Case 1), it suffices to show that $ \int |q_p- q| dx \to 0$. Since $\xi_p$ converge in distribution to $\mathcal{N}(0,1)$, we know that for any finite $R$, $ \int_{|x|<R} |q_p(x)- q(x)| dx \to 0$. Thus, it suffices to show that \[
 \int_{|x| > R} q_p(x) dx \to 0, \quad R\to \infty
\]  uniformly in $p$. This follows from the large deviation bound that is given in Lemma \ref{lemma:xipLargeDeviation}.
\end{proof}

\section{Lemma in Sec. 4}\label{sec:A4} 

\begin{lemma}\label{lemma:4thmoment_lge2}
Let $X_{1},X_{2},X_{3},X_{4}$ be i.i.d distributed as $\mathcal{N}(0,p^{-1}I_{p})$, and $P_{l,p}(x)$ is the degree-$l$ Hermite-like polynomial as defined in Sec. \ref{subsec:hermite}, $l\ge2$. Then
\[
\mathbb{E}P_{l,p}(\sqrt{p}\xi_{12})P_{l,p}(\sqrt{p}\xi_{23})P_{l,p}(\sqrt{p}\xi_{34})P_{l,p}(\sqrt{p}\xi_{41})=\mathcal{O}_{l}(1)p^{-2},
\]
where $\xi_{ij}=X_{i}^{T}X_{j}$.
\end{lemma}
\begin{proof}
Write 
\[
\xi_{12}=|X_{1}|\eta_{2},\,\xi_{14}=|X_{1}|\eta_{4},\,\xi_{23}=\eta_{2}\eta_{3}+\tilde{\xi}_{23},\,\xi_{34}=\eta_{3}\eta_{4}+\tilde{\xi}_{34}
\]
where $\eta_{2},\eta_{3},\eta_{4}$ are i.i.d distributed as $\mathcal{N}(0,p^{-1})$, and $|X_{1}|$, $\eta_{i}$'s and $\tilde{\xi}_{23}$,$\tilde{\xi}_{34}$ are jointly independent. Since $P_{l,p}(x)$ is a polynomial of degree $l$,
\begin{eqnarray*}
P_{l,p}(x_{1}+x_{2}) 
& = & P_{l,p}(x_{2})+x_{1}P_{l,p}^{'}(x_{2})+ P_{(2)}(x_{1},x_{2}) \\
P_{(2)}(x_{1},x_{2}) 
& = & \sum_{k=2}^{l}\frac{P_{l,p}^{(k)}(x_{2})}{k!}x_{1}^{k},
\end{eqnarray*}
thus
\begin{eqnarray*}
  & & \mathbb{E}P_{l,p}(\sqrt{p}\xi_{12})P_{l,p}(\sqrt{p}\xi_{23})P_{l,p}(\sqrt{p}\xi_{34})P_{l,p}(\sqrt{p}\xi_{41}) \\
 &= & \mathbb{E}P_{l,p}(\sqrt{p}|X_{1}|\eta_{2})P_{l,p}(\sqrt{p}|X_{1}|\eta_{4})\\
 & & \quad \cdot \left[P_{l,p}(\sqrt{p}\tilde{\xi}_{23})+\sqrt{p}\eta_{2}\eta_{3}P_{l,p}^{'}(\sqrt{p}\tilde{\xi}_{23})+P_{(2)}(\sqrt{p}\eta_{2}\eta_{3},\sqrt{p}\tilde{\xi}_{23})\right]\\
 &  & \quad \cdot\left[P_{l,p}(\sqrt{p}\tilde{\xi}_{34})+\sqrt{p}\eta_{3}\eta_{4}P_{l,p}^{'}(\sqrt{p}\tilde{\xi}_{34})+P_{(2)}(\sqrt{p}\eta_{3}\eta_{4},\sqrt{p}\tilde{\xi}_{34})\right].
\end{eqnarray*}

Notice that $\mathbb{E}\eta_{3}=0$, thus it suffices to show that
\begin{eqnarray*}
\mathbb{E}P_{l,p}(\sqrt{p}|X_{1}|\eta_{2})P_{l,p}(\sqrt{p}|X_{1}|\eta_{4})P_{l,p}(\sqrt{p}\tilde{\xi}_{23})P_{l,p}(\sqrt{p}\tilde{\xi}_{34}),\\
\mathbb{E}\sqrt{p}\eta_{2}P_{l,p}(\sqrt{p}\eta_{2}|X_{1}|)\cdot\sqrt{p}\eta_{4}P_{l,p}(\sqrt{p}|X_{1}|\eta_{4})\cdot\eta_{3}^{2}\cdot P_{l,p}^{'}(\sqrt{p}\tilde{\xi}_{23})P_{l,p}^{'}(\sqrt{p}\tilde{\xi}_{34}),\\
\mathbb{E}\sqrt{p}\eta_{2}P_{l,p}(\sqrt{p}\eta_{2}|X_{1}|)\cdot P_{l,p}^{'}(\sqrt{p}\tilde{\xi}_{23})\cdot P_{l,p}(\sqrt{p}|X_{1}|\eta_{4})\eta_{3}P_{(2)}(\sqrt{p}\eta_{3}\eta_{4},\sqrt{p}\tilde{\xi}_{34}),\\
\mathbb{E}P_{l,p}(\sqrt{p}|X_{1}|\eta_{2})P_{l,p}(\sqrt{p}|X_{1}|\eta_{4})P_{(2)}(\sqrt{p}\eta_{2}\eta_{3},\sqrt{p}\tilde{\xi}_{23})P_{(2)}(\sqrt{p}\eta_{3}\eta_{4},\sqrt{p}\tilde{\xi}_{34})
\end{eqnarray*}
are all at bounded by $\mathcal{O}_{l}(1)p^{-2}$. This can be done by making use of the facts that $\mathbb{E}|X_{1}|^{2m}=1+\mathcal{O}_{m}(1)p^{-1}$ (Eq. \ref{eq:E|X|2m}), and that the differences between the coefficients of $P_{l,p}(x)$ and those of $h_{l}(x)$ are $\mathcal{O}_{l}(1)p^{-1}$ (Lemma \ref{lemma:Plpmatchhl}). The condition $l \ge 2$ is needed to guarantee that $P_{l,p}(x)$ and $x$ are asymptotically orthogonal.
\end{proof}


\begin{proof}[Proof of Lemma \ref{lemma:m_A-m_B}]
We have
\begin{align*}
& m_{A}(z)-m_{B}(z) \\
& = \frac{1}{n}\left( \Tr ((A-zI)^{-1}) - \Tr ((B-zI)^{-1}) \right) \\
& = \frac{1}{n} \Tr ((A-zI)^{-1}(B-A) (B-zI)^{-1}),
\end{align*}
thus
\begin{align*}
& \mathbb{E} |m_{A}(z)-m_{B}(z)|^2  \\
& = \mathbb{E}\frac{1}{n^2}(\Tr ((A-zI)^{-1}(B-A) (B-zI)^{-1}))^2 \\
& \le \mathbb{E} \frac{1}{n^2} \Tr (((B-zI)^{-1}(A-zI)^{-1})^2 ) \Tr ((B-A)^2)  \\
& \le \mathbb{E} \frac{1}{n^2} \frac{n}{v^4} \sum_{i,j=1}^{n}|A_{ij}-B_{ij}|^{2},\\
& \le \frac{1}{v^4 n} \sum_{i,j=1}^{n}\mathbb{E}(f_{A}(X_{i}^{T}X_{j};p)-f_{B}(X_{i}^{T}X_{j};p))^{2}\\
& \le \frac{1}{v^4 n} n^{2}p^{-1}\epsilon ={\cal O}(1)\epsilon. \qedhere
\end{align*}


\end{proof}


\begin{lemma}\label{pf:omega_delta_p-9}
Let $\Omega_\delta$ be defined as in Eq. (\ref{eq:def.Omega_delta}),
\[
\Pr(\Omega_{\delta}^{c})\le{\cal O}(1)p^{-7}.
\]
\end{lemma}
\begin{proof} For $\eta_{i}$ we have the concentration inequality Eq. (\ref{eq:eta_delta}); For each $\tilde{\xi}_{ij}$, we write it as 
\[
\tilde{\xi}_{ij}=|\tilde{X}_{i}|\tilde{\eta}_{ij},
\]
where $\tilde{\eta}_{ij}$ has marginal distribution $\mathcal{N}(0,p^{-1})$ and is independent of $|\tilde{X}_{i}|$. With inequality Eq. (\ref{bound.|Xn|^1-1isdelta_p-9}) which also holds for $|X_{i}|$ in place of $|X_{n}|$, we have
\begin{align*}
\Pr [ |\tilde{X}_{i}||\tilde{\eta}_{ij}|  >\delta ]
& \le \Pr \left[ |\tilde{X}_{i}|^{2}>1+\sqrt{\frac{40\ln p}{p}} \right] \\
& \quad + \Pr \left[|\tilde{X}_{i}||\tilde{\eta}_{ij}|>\delta,|\tilde{X}_{i}|^{2}<1+\sqrt{\frac{40\ln p}{p}}\right]\\
 & \le p^{-9}+\Pr[|\tilde{\eta}_{ij}|>\frac{\delta}{1.01}]\\
 & \le p^{-9}+\frac{1}{\sqrt{2}}p^{-9},
\end{align*}
 thus 
\[
\Pr[|\tilde{\xi}_{ij}|>\delta]<{\cal O}(1)p^{-9}.
\]

Then, a union bound gives
\begin{align*}
\Pr(\Omega_{\delta}^{c}) 
& \le (n-1)\Pr[|\eta_{i}|>\delta]  \\
& \quad + \frac{(n-1)(n-2)}{2}\Pr[|\tilde{\xi}_{ij}|>\delta]+\Pr[\left||X_{n}|^{2}-1\right|>\sqrt{2}\delta]\\
 & \le{\cal O}(1)p^{-9}+{\cal O}(1)p^{-7}+p^{-9}={\cal O}(1)p^{-7}.\qedhere
\end{align*}
\end{proof}

\begin{lemma}\label{lemma:step2_E|r_2|}
Notation as in Sec. \ref{subsec:bigproof}. $r_2$ defined in Eq. (\ref{eq:step2_r1r2}) satisfies \[
\mathbb{E}|r_{2}|\cdot\mathbf{1}_{\Omega_{\delta}}\le{\cal O}_{L}(1)M^{2}p^{-1/2}.
\]
\end{lemma}
\begin{proof}
From Eq. (\ref{eq:step_2r21r22}), firstly, 
\[
r_{2,1} = f_{(2)}^{T}(\tilde{A}^{(n)}-zI_{n-1})^{-1}(|X_{n}|\eta)
\]
satisfies $\mathbb{E}|r_{2,1}|\le{\cal O}_{L}(1)p^{-1/2}$ by a moment bound: 
recall the definition of $\xi_{in}$ as in Eq. (\ref{eq:Xn_2}), and that $f_{>1}(\xi)$ is a linear combination of rescaled and renormalized Hermite-like polynomials of degree $\ge 2$.
Also, $\mathbb{E}|X_n|^{2m} = 1+{\cal O}_{m}(1)p^{-1} $ (Eq. (\ref{eq:E|X|2m})), and $|X_n|$ is independent from $\eta_i$'s and $\tilde{X}_i$'s. Denote $\tilde{B} = (\tilde{A}^{(n)}-zI_{n-1})^{-1}$. By taking expectation over $|X_n|$ first and then over $\eta_i$'s, we have
\begin{align*}
\mathbb{E}|r_{2,1}|^{2} & =\mathbb{E}\left|\sum_{i_{1},i_{2}=1}^{n-1}f_{>1}(\xi_{i_{1}n})\xi_{i_{2}n}\tilde{B}_{i_{1}i_{2}}\right|^{2}\\
 & =\mathbb{E}\sum_{i_{1},i_{2}}\sum_{i_{1}^{'},i_{2}^{'}}f_{>1}(\xi_{i_{1}n})\xi_{i_{2}n}f_{>1}(\xi_{i_{1}^{'}n})\xi_{i_{2}^{'}n}\tilde{B}_{i_{1}i_{2}}\overline{\tilde{B}_{i_{1}^{'}i_{2}^{'}}}\\
 & = \{i_{1}=i_{2}=i_{1}^{'}=i_{2}^{'}\}+\{i_{1},i_{2}=i_{1}^{'}=i_{2}^{'},\text{or \ensuremath{i_{1}^{'}}as \ensuremath{i_{1}}}\}  \\
 & \quad + \{i_{2}=i_{2}^{'},i_{1},i_{1}^{'}\}+\{i_{1}=i_{2},i_{1}^{'}=i_{2}^{'},\text{or \ensuremath{i_{1}^{'}}as \ensuremath{i_{1}}}\}+\{i_{1}=i_{1}^{'},i_{2}=i_{2}^{'}\}\\
 & ={\cal O}_{L}(1)p^{-1}+\nu_{>1, p} p^{-2}\mathbb{E}\mathbf{Tr}(\overline{\tilde{B}}^{T}\tilde{B})\\
 & \le{\cal O}_{L}(1)p^{-1}+\mathcal{O}(1)\cdot p^{-2}\frac{n}{v^{2}}={\cal O}_{L}(1)p^{-1},
\end{align*}
where by $\{i_1, i_2 = i_1^{'}=i_2^{'}\}$ we denote the terms in summation where the last three indices take the same value while $i_1$ is distinct from them, and similar for others. 

Secondly, 
\begin{align*}
r_{2,2} & =(f_{(2)}^{T}(\tilde{A}^{(n)}-zI_{n-1})^{-1}(|X_{n}|\eta))(a_{1}(p)\eta^{T}(\hat{A}^{(n)}-zI_{n-1})^{-1}\eta)\\
 & =r_{2,1}(a_{1}(p)\eta^{T}(\hat{A}^{(n)}-zI_{n-1})^{-1}\eta),
\end{align*}
where 
\begin{align*}
|a_{1}(p)\eta^{T}(\hat{A}^{(n)}-zI_{n-1})^{-1}\eta|\cdot\mathbf{1}_{\Omega_{\delta}} & \le {\cal O}(1) s((\hat{A}^{(n)}-zI_{n-1})^{-1})||\eta||^{2}\cdot\mathbf{1}_{\Omega_{\delta}}\\
 & \le {\cal O}(1)M^{2}={\cal O}(1)M^{2},
\end{align*}
thus 
\[
\mathbb{E}|r_{2,2}|\cdot\mathbf{1}_{\Omega_{\delta}}\le{\cal O}(1)M^{2}\mathbb{E}|r_{2,1}|\le{\cal O}(1)M^{2}\cdot{\cal O}_{L}(1)p^{-1/2}={\cal O}_{L}(1)M^{2}p^{-1/2}.
\]
Then 
\[
\mathbb{E}|r_{2}|\cdot\mathbf{1}_{\Omega_{\delta}}\le {\cal O}(1) (\mathbb{E}|r_{2,1}|+\mathbb{E}|r_{2,2}|\cdot\mathbf{1}_{\Omega_{\delta}})\le{\cal O}_{L}(1)M^{2}p^{-1/2}.\qedhere
\]
\end{proof}

\begin{lemma}\label{lemma:kmoment_S^p-1}
Notations as in Sec. \ref{subsec:XiS},\[
\mathbb{E}(\xi^{'}_p)^{k}=\begin{cases}
(k-1)!!+{\cal O}_{k}(1)p^{-1}, & k\:\text{even;}\\
0, & k\:\text{odd}.
\end{cases}\]
\end{lemma}
\begin{proof}
The odd moments vanish since the distribution of $\xi_p'$ is symmetric with respect to 0. For even moments, let $k=2m$. Let $\xi_p=\sqrt{p}X^{T}Y$ where $X$ and $Y$ are i.i.d $\mathcal{N}(0,p^{-1}I_{p})$, and we have that $\xi_p$ and $\xi_p'|X||Y|$ observe the same probability distribution. Notice that $\xi_p'$, $|X|$ and $|Y|$ are independent, so 
\[
\mathbb{E}\xi_p^{2m}=\mathbb{E}|X|^{2m}\mathbb{E}|Y|^{2m}\mathbb{E}(\xi_p')^{2m}=(\mathbb{E}|X|^{2m})^{2}\mathbb{E}(\xi_p')^{2m}.
\]
By Eq. (\ref{eq:Exi^k}), to show the claim it suffices to show that $\mathbb{E}|X|^{2m}=1+{\cal O}_{m}(1)p^{-1}$. To do this, define
\[
r=|X|^{2}-1=\sum_{j=1}^{p}\left(X_{j}^{2}-\frac{1}{p}\right).
\]
Due to the mutual independence of the $X_{j}$'s, the odd moments of $r$ vanish; $\mathbb{E}r^{2}=2p^{-1}$, and generally for even
$l$ 
\[ 
\mathbb{E}\left( \sqrt{\frac{p}{2}}\,r \right)^{l}=(l-1)!!+{\cal O}_{l}(1)p^{-1},
\]
so $ \mathbb{E}r^{l}={\cal O}_{l}(1)p^{-l/2}$. Then 
\begin{align}
\mathbb{E}|X|^{2m} & =\mathbb{E}(1+r)^{m} \nonumber \\ 
 & =1+\sum_{l=2,l\text{ even}}^{m}c(l,m)\mathbb{E}r^{l} \nonumber \\
 & =1+\sum_{l=2,l\text{ even}}^{m}c(l,m){\cal O}_{l}(1)p^{-l/2} \nonumber\\
 & =1+{\cal O}_{m}(1)p^{-1}.\label{eq:E|X|2m}
 \qedhere
\end{align}
\end{proof}

\bibliography{cheng_RMT}{}
\bibliographystyle{imsart-nameyear}

\end{document}